\documentclass[11pt]{amsart}

\usepackage{amsmath,amssymb}
\usepackage{url}
\usepackage{booktabs}
\usepackage{epsfig}
\usepackage{fullpage}

\theoremstyle{plain}
\newtheorem{theorem}{Theorem}[section]
\newtheorem{lemma}[theorem]{Lemma}

\newtheorem{prop}[theorem]{Proposition}
\newtheorem{conj}[theorem]{Conjecture}

\theoremstyle{definition}

\newtheorem{example}[theorem]{Example}
\newtheorem{remark}[theorem]{Remark}

\newcommand{\dqbin}[3]{\displaystyle\genfrac{[}{]}{0pt}{}{#1}{#2}_{#3}}
\newcommand{\tqbin}[3]{\textstyle\genfrac{[}{]}{0pt}{}{#1}{#2}_{#3}}

\DeclareMathOperator{\sw}{sw}
\DeclareMathOperator{\wt}{wt}

\newcommand{\ZZ}{\mathbb{Z}}

\newcommand{\Z}{\mathbb{Z}}

\newcommand{\FCHL}{\mathrm{FCHL}}

\newcommand{\N}{\mathrm{N}}
\newcommand{\D}{\mathrm{D}}
\newcommand{\E}{\mathrm{E}}
\newcommand{\HL}{\mathrm{HL}}
\newcommand{\EHKK}{\mathrm{EHKK}}
\newcommand{\LW}{\mathrm{LW}}
\newcommand{\NE}[1]{\mathrm{#1}}

\newcommand{\bounce}{\mathsf{bounce}}

\newcommand{\Area}{\mathsf{Area}}
\newcommand{\area}{\mathsf{area}}

\newcommand{\dinv}{\mathsf{dinv}} 

\newcommand{\ml}{\mathsf{ml}}

\newcommand{\zetamap}{\mathsf{zeta}}

\newcommand{\wdR}{\mathcal{R}^{\mathrm{word}}}   
\newcommand{\wdD}{\mathcal{D}^{\mathrm{word}}}   
\newcommand{\patR}{\mathcal{R}^{\mathrm{path}}} 
\newcommand{\patD}{\mathcal{D}^{\mathrm{path}}} 
\newcommand{\ptnR}{\mathcal{R}^{\mathrm{ptn}}} 
\newcommand{\ptnD}{\mathcal{D}^{\mathrm{ptn}}} 

\newcommand{\drewnu}{f}
\newcommand{\drewrho}{g} 
\newcommand{\newnu}{\tilde{f}}
\newcommand{\newrho}{\tilde{g}} 
\newcommand{\gmmap}{\mathrm{G}}

\DeclareMathOperator{\mkptn}{\textsc{ptn}}
\DeclareMathOperator{\mkwd}{\textsc{word}}
\DeclareMathOperator{\mkpath}{\textsc{path}} 
\DeclareMathOperator{\flip}{\mathsf{flip}}
\DeclareMathOperator{\rev}{\mathsf{rev}}
\DeclareMathOperator{\id}{\mathsf{id}}

\makeatletter
\def\imod#1{\allowbreak\mkern10mu({\operator@font mod}\,\,#1)}
\makeatother

\begin{document}

\title{Sweep maps: A continuous family of sorting algorithms}

\date{\today}

\author{Drew Armstrong}
\address{Dept. of Mathematics\\
  University of Miami \\
  Coral Gables, FL 33146}
\email{armstrong@math.miami.edu}

\author{Nicholas A. Loehr}
\address{Dept. of Mathematics\\
  Virginia Tech \\
  Blacksburg, VA 24061-0123 \\
and Mathematics Dept. \\
United States Naval Academy \\
Annapolis, MD 21402-5002}
\email{nloehr@vt.edu}

\thanks{This work was partially supported by a grant from the Simons
 Foundation (\#244398 to Nicholas Loehr).}

\author{Gregory S. Warrington}
\address{Dept. of Mathematics and Statistics\\
  University of Vermont \\
  Burlington, VT 05401}
\email{gregory.warrington@uvm.edu}

\thanks{Third author supported in part by National Science Foundation
  grant DMS-1201312.}

\vspace*{.3in}

\begin{abstract}
 We define a family of maps on lattice paths, called \emph{sweep maps},
 that assign levels to each step in the path and sort steps according to
 their level. Surprisingly, although sweep maps act by sorting, they
 appear to be bijective in general. The sweep maps give
 concise combinatorial formulas for the $q,t$-Catalan numbers, the
 higher $q,t$-Catalan numbers, the $q,t$-square numbers, and many
 more general polynomials connected to the nabla operator and rational
 Catalan combinatorics. We prove that many algorithms that have appeared
 in the literature (including maps studied by 
 Andrews, Egge, Gorsky, Haglund, Hanusa, Jones, Killpatrick, 
 Krattenthaler, Kremer, Orsina, Mazin, Papi, Vaill{\'e}, and the
 present authors) are all special cases of the sweep maps or their 
 inverses. The sweep maps provide a very simple unifying framework
 for understanding all of these algorithms.  
 We explain how inversion of the sweep map (which is an open
 problem in general) can be solved in known special cases by finding a 
 ``bounce path'' for the lattice paths under consideration.
 We also define a generalized sweep map acting on words over arbitrary 
 alphabets with arbitrary weights, which is also conjectured to be bijective.  
\end{abstract}

\maketitle

\section{Introduction}
\label{sec:intro}

This paper introduces a family of sorting maps on words that we call
\emph{sweep maps}.  In its simplest form, a sweep map $\sw_{r,s}$ uses
coprime parameters $r$ and $s$ to associate a \emph{level} $l_i$ to
each letter $w_i$ in a word $w=w_1w_2\cdots w_n$ consisting of $|s|$
copies of the letter $\N$ and $|r|$ copies of the letter $\E$.  (Note
that $r$ or $s$ may be negative.)  This assignment is done
recursively: use the convention that $l_0=0$; for $i\geq 1$ we set
$l_i = l_{i-1}+r$ if $w_i=\N$ and $l_i=l_{i-1}+s$ if $w_i=\E$.  The
word $\sw_{r,s}(w)$ is then obtained by sorting the letters in $w$
according to level, starting with $-1,-2,-3,\ldots$, then continuing
with $\ldots,2,1,0$.  Figure~\ref{fig:sweepmap1} provides an
example of $\sw_{5,-3}$ acting on the word $w=\NE{ENEENNEE}$.  (Here
we have identified $w$ with a lattice path in the plane: each $\N$
corresponds to a unit-length north step, while each $\E$ corresponds to
a unit-length east step.)

\begin{figure}[htbp]
  \centering 
      {\scalebox{.4}{\includegraphics{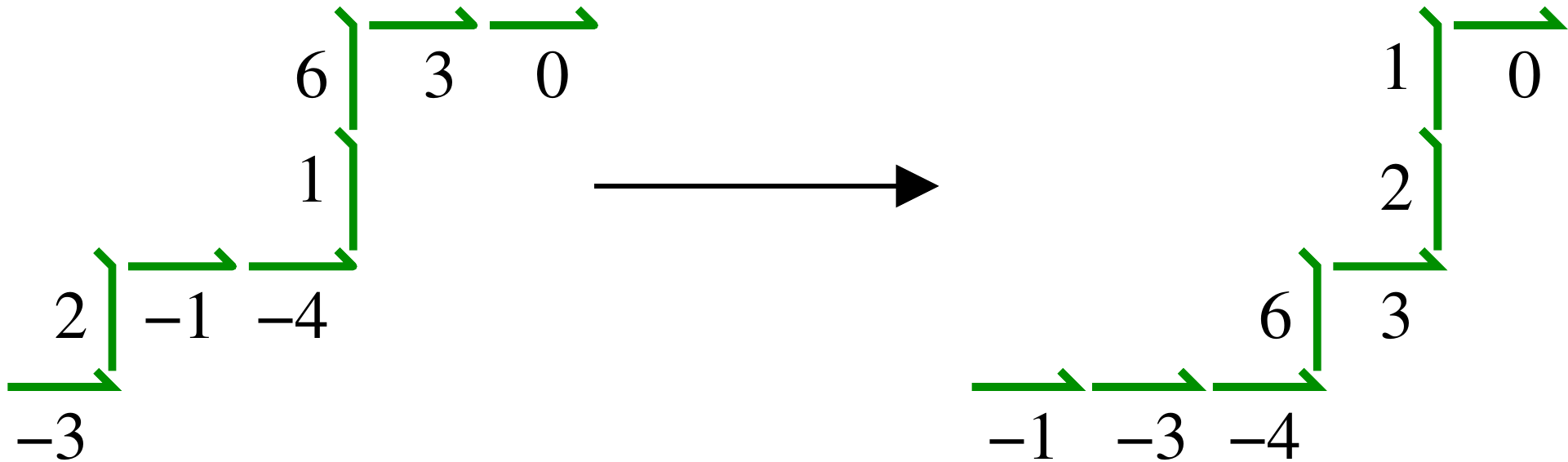}}}
      \caption{The action of the sweep map $\sw_{5,-3}$ on the word
        $w=\NE{ENEENNEE}$.  Next to each step in $w$ is written
        its level, $l_i$.  Each step in $\sw_{5,-3}(w)$ has been
        labeled by the level of the corresponding step in $w$.}
  \label{fig:sweepmap1}
\end{figure}

Surprisingly, even though sweep maps act by sorting, they are
(apparently) bijective.  The reader may find it useful to check this
bijectivity by hand for $\sw_{3,-2}$ acting on the set of all lattice
paths from $(0,0)$ to $(3,2)$.  
As detailed in Conjecture~\ref{conj:gen-sweep}, bijectivity seems to
hold even for the general sweep maps over arbitrary alphabets with
arbitrary weights, described in Section~\ref{subsec:gen}.  
The bijectivity of the general sweep maps appears
to be a very subtle and difficult fact.

\begin{remark}
  The order in which the levels are traversed is a key ingredient to
  bijectivity.  For example, in the case of $r=3$, $s=-2$,
  if we scan levels in the order $k=\ldots,2,1,0,-1,-2,\ldots$, both
  of the paths $\NE{NENEE}$ and $\NE{NEENE}$ map to $\NE{NNEEE}$.
\end{remark}

The sweep maps encode complex combinatorial information related to
$q,t$-Catalan numbers, the Bergeron-Garsia nabla operator, and other
constructions arising in representation theory, algebraic geometry,
and symmetric functions. Researchers have discovered special cases of
the sweep map in many different guises over the last fifteen years or
so.  One of the goals of this paper is to present a unifying framework
for all of this work.  In~\cite{loehr-mcat,loehr-thesis,loehr-trapz},
Loehr introduced bijections on $m$-Dyck paths, as well as
generalizations to lattice paths contained in certain trapezoids, that
turn out to be special cases of the sweep map.  The bijection in the
case $m=1$ also appears in a paper of Haglund and Loehr~\cite{HL-park}
and is foreshadowed by a counting argument in Haglund's seminal paper
on $q,t$-Catalan numbers~\cite[proof of Thm. 3.1]{Hag-bounce}.  The
inverse bijection in the case $m=1$ appears even earlier, where it was
used by Andrews et al.~\cite{AKOP-lie} in their study of
$ad$-nilpotent $\mathfrak{b}$-ideals in the Lie algebra
$\mathfrak{sl}(n)$.  (See also~\cite{vaille}.)  More recently, special
cases of the sweep map have arisen while studying lattice paths in
squares~\cite{LW-square}; partition statistics~\cite{LW-ptnid};
simultaneous core partitions~\cite{AHJ}; and compactified
Jacobians~\cite{GM-jacI,GM-jacII}.  We discuss a number of these
connections in more detail in Section~\ref{sec:alg-sweep}.

We suspect that to the typical mathematician, the most
interesting question regarding the sweep maps is whether they are
bijective (as conjectured in Conjecture~\ref{conj:gen-sweep}).  For a
researcher interested in the $q,t$-Catalan numbers, however, of
comparable interest is the connection between the sweep maps and
statistics on lattice paths such as $\area$, $\bounce$ and $\dinv$.
Since shortly after Haiman's introduction of $\dinv$, it has been known
that a ``zeta map'' takes $\dinv$ to $\area$ to $\bounce$.  One point
of view, then, is that rather than having three statistics on Dyck
paths, we have one statistic --- $\area$ --- and a sweep map.  Many
polynomials related to the $q,t$-Catalan numbers can be defined using only an
``area'' and an appropriate sweep map.  That these polynomials are
jointly symmetric (conjecturally) supports the utility of this view (see
Section~\ref{sec:area-qtcat}).

The structure of this paper is as follows.  Section~\ref{sec:basic}
introduces the necessary background on lattice paths.  We then define
sweep maps and present Conjecture~\ref{conj:gen-sweep} on their
bijectivity in Section~\ref{sec:intro-sweep}.
Section~\ref{sec:alg-sweep} reviews various algorithms that have
appeared in the literature that are equivalent to special cases of 
the sweep map, while
Section~\ref{sec:invert-sweep} describes how to invert these maps
(when known).  Finally, Section~\ref{sec:area-qtcat} shows how the
sweep maps may be used to give concise combinatorial formulas for the
higher $q,t$-Catalan numbers and related polynomials formed
by applying the nabla operator to appropriate symmetric functions and
then extracting the coefficient of $s_{(1^n)}$. 

An extended abstract of this paper appears as~\cite{ALW-sweep-fpsac}.


\section{Partitions, Words, and Lattice Paths}
\label{sec:basic}

This section introduces our basic conventions regarding 
partitions, words, and lattice paths.
Integer parameters $a$ and $b$ will serve to restrict our attention to such
objects fitting within a rectangle of height $a$ and width $b$.
Integer parameters $r$ and $s$ will be used to assign a ``level'' to the 
components of the various objects considered.  Of particular interest is the 
case of $r=b$ and $s=-a$.  The constraint of $\gcd(a,b)=1$ arises naturally 
in some particular sweep maps such as the map due to Armstrong, Hanusa
and Jones~\cite{AHJ} and the map due to Gorksy and
Mazin~\cite{GM-jacI,GM-jacII} (see Sections~\ref{subsec:zetamap}
and~\ref{subsec:gorsky-mazin}, respectively).

Let $a,b \in \mathbb{Z}_{\geq 0}$.  Integer partition diagrams with at most
$a$ parts and largest part at most $b$ (drawn according to the English 
convention) fit in the rectangle with vertices $(0,0)$, $(b,0)$,
$(0,a)$ and $(b,a)$.  We denote the set of such partitions 
(identified with their diagrams, which are collections
of unit squares in the first quadrant) by $\ptnR = \ptnR(a,b)$.  Let
$\{\N,\E\}^*$ denote the set of all words $w = w_1w_2\cdots w_n$, $n\geq
0$, for which each $w_j \in \{\N,\E\}$, and let $\wdR = \wdR(\N^a\E^b)$
denote the subset of words consisting of $a$ copies of $N$ and $b$
copies of $E$.  Finally, let $\patR = \patR(\N^a\E^b)$ denote the set of
lattice paths from $(0,0)$ to $(b,a)$ consisting of $a$ unit-length
north steps and $b$ unit-length east steps.

There are natural bijections among the three sets $\ptnR$, $\wdR$ and
$\patR$.  Each word $w\in \wdR$ encodes a lattice path in $\patR$ by
letting each $\E$ correspond to an east step and each $\N$ correspond to
a north step.  The frontier of a partition $\pi\in \ptnR$ also
naturally encodes a path in $\patR$.  We write $\mkwd(P)$ or
$\mkwd(\pi)$ for the word associated with a path $P$ or a partition $\pi$,
respectively.  Operators $\mkpath$ and $\mkptn$ are defined analogously.
For example, taking $a=3$ and $b=5$, the path $P$ shown on the left
in Figure~\ref{fig:sweepmap1} has $\mkwd(P)=\NE{ENEENNEE}$ and
$\mkptn(P)=(3,3,1)$. For the word $w=\NE{EEENENNE}$, $\mkpath(w)$
is the path shown on the right in Figure~\ref{fig:sweepmap1}, whereas
$\mkptn(w)=(4,4,3)$.

Let $r,s \in \Z$.  We assign a \textbf{level} to each square of a
partition and each step of a path in the following manner.  First,
assign to each lattice point $(x,y)$ the
\textbf{$\boldsymbol{(r,s)}$-level} $ry+sx$.  Assign to each lattice
square the level of its southeast corner.  We will have occasion to
consider two different conventions for associating levels to north and
east steps of paths.  For the \textbf{east-north (E-N)} convention, each
east (resp. north) step inherits the level of its eastern (resp. northern)
endpoint.  For the \textbf{west-south (W-S)} convention, each east (resp. north)
step inherits the level of its western (resp. southern) endpoint.
Figure~\ref{fig:labeling} illustrates the various levels relevant to
the word $\NE{NNENE}$ for $r=8$ and $s=-5$.

\begin{figure}[htbp]
  \centering 
      {\scalebox{.6}{\includegraphics{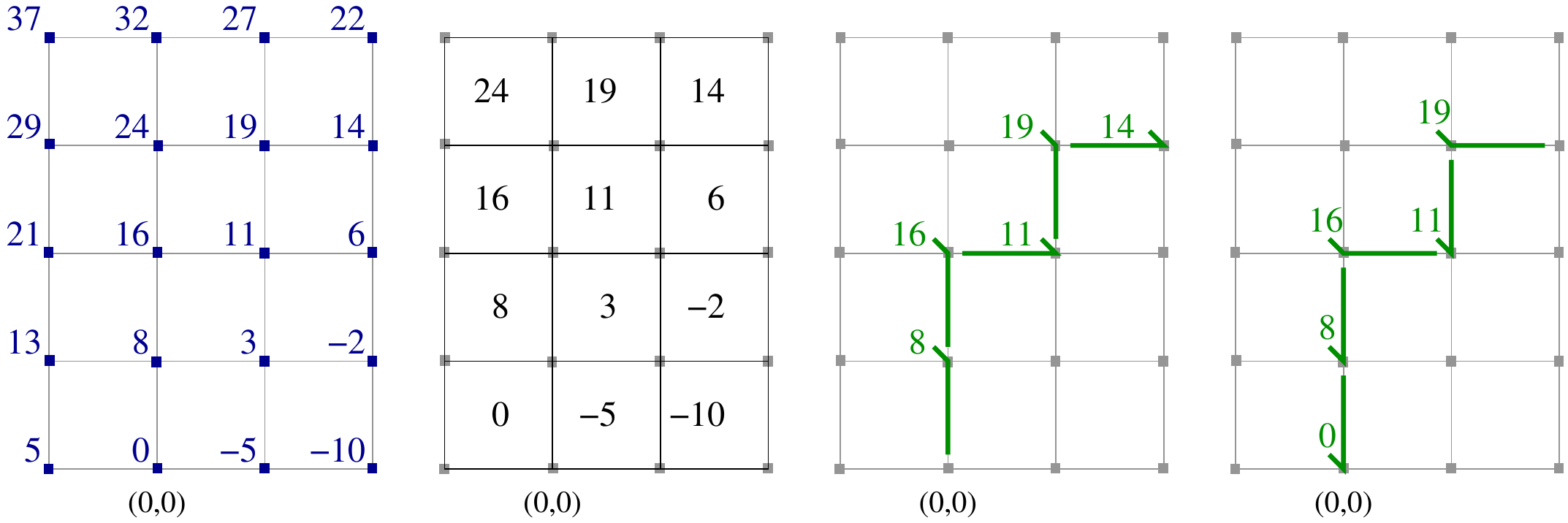}}}
      \caption{Illustration of level-assignment conventions for lattice
        points, squares, path steps with the east-north convention, and
        path steps with the west-south convention for the case of $r=8$
        and $s=-5$.}
  \label{fig:labeling}
\end{figure}

Let $\patD_{r,s}(\N^a\E^b)$ denote the set of lattice paths 
in $\patR(\N^a\E^b)$ whose 
steps all have nonnegative $(r,s)$-levels under the E-N convention.
We call these paths \textbf{$\boldsymbol{(r,s)}$-Dyck paths}.
A word $w\in\wdR(\N^a\E^b)$ is an \textbf{$\boldsymbol{(r,s)}$-Dyck word}
iff $\mkpath(w)$ is an $(r,s)$-Dyck path;
 let $\wdD_{r,s}(\N^a\E^b)$ denote the set of such words.
A partition $\pi\in\ptnR(a,b)$ is an \textbf{$\boldsymbol{(r,s)}$-Dyck 
partition} iff $\mkpath(\pi)$ is an $(r,s)$-Dyck path;
 let $\ptnD_{r,s}(a,b)$ denote the set of such partitions.

\section{The Sweep Map}
\label{sec:intro-sweep}

We begin in Section~\ref{subsec:def-sweep} by giving an algorithmic
description of the basic sweep maps for words over the alphabet $\{N,E\}$.
Some minor variations are presented in Sections~\ref{subsec:irrat-sweep}
and~\ref{subsec:minor-var}.
We then present a general sweep map in Section~\ref{subsec:gen} 
that acts on words over any alphabet with arbitrary weights
assigned to each letter.

\subsection{The Basic Sweep Map}
\label{subsec:def-sweep}

Let $r,s\in \Z$.  We first describe the
\textbf{$\boldsymbol{(r,s)^{-}}$-sweep map},
$\sw^{-}_{r,s}:\{\N,\E\}^*\rightarrow\{\N,\E\}^*$.  Given $w \in
\{\N,\E\}^*$, assign levels using the east-north convention applied to
$\mkpath(w)$.  Define a word $y=\sw^{-}_{r,s}(w)$ by the following
algorithm.  Initially, $y$ is the empty word. For $k=-1,-2,-3,\ldots$
and then for $k=\ldots,3,2,1,0$, scan $w$ from right to left.
Whenever a letter $w_i$ is encountered with level equal to $k$, append
$w_i$ to $y$.  The \textbf{$\boldsymbol{(r,s)^{+}}$-sweep map}
$\sw^{+}_{r,s}$ is defined the same way as $\sw^{-}_{r,s}$, except
that: the value $0$ is the first value of $k$ used rather than the
last; and for each value of $k$, $w$ is scanned from left to right.
Figure~\ref{fig:negsweep-ex} illustrates the action of both
$\sw^{+}_{3,-2}$ and $\sw^{-}_{3,-2}$ on a path in
$\patR(\N^8\E^{10})$.  
We define the action of $\sw^{-}_{r,s}$ on a partition $\pi$ as
$\mkptn(\sw_{r,s}^{-}(\mkwd(\pi)))$ and the action of $\sw_{r,s}^{-}$ on a
path as $\mkpath(\sw_{r,s}^{-}(\mkwd(P)))$; similarly for $\sw^{+}_{r,s}$.

\begin{figure}[htbp]
  \centering 
      {\scalebox{.3}{\includegraphics{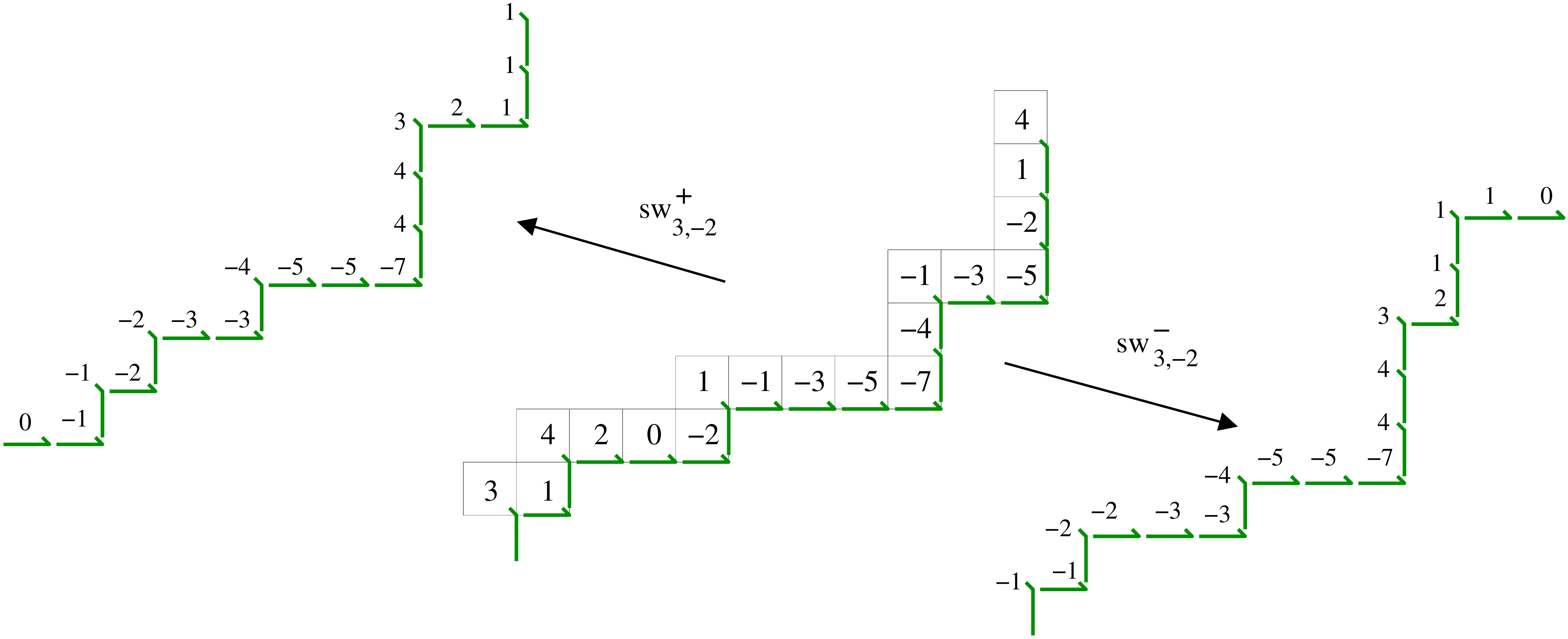}}}
      \caption{Images of $\sw^{-}_{3,-2}$ and $\sw^{+}_{3,-2}$ acting on
        a path in $\patR(\N^8\E^{10})$.}
  \label{fig:negsweep-ex}
\end{figure}

Geometrically, we think of each step in $\mkpath(w)$ as a \emph{wand} whose
\emph{tip} is located at the lattice point at the end of each step.  
This visual description reminds us that the maps $\sw^{\pm}_{r,s}$
are assigning levels to steps in the path using the east-north convention.
For $r>0$ and $s<0$, $\sw^-_{r,s}$ acts 
by scanning southwest along each diagonal line
$ry+sx=k$ (taking $k$'s in the appropriate order) 
and ``sweeping up'' the wands whose
tips lie on each of these diagonals. The wands are laid out in the
order in which they were swept up to produce the output lattice path.
The labels on the wand tips are \emph{not} part of the final output.  
It is clear from this description, or from the original definition, that the
sweep map depends only on the ``slope'' $-s/r$; i.e., 
$\sw_{r,s}^{\pm}=\sw_{rm,sm}^{\pm}$ for every positive integer $m$.

\begin{remark}
 The $\area^*_{r,s}$ statistic introduced in
 Section~\ref{subsec:area-stat} is computed using the $(r,s)$-levels
 of the steps in $w$, which depend not just on the ratio $-s/r$ but 
 on the specific values of $r$ and $s$.
 So it is not always safe to assume that $r$ and $s$ are relatively prime.
\end{remark}

Since the sweep maps rearrange the letters in the input word 
$w$, it is immediate that both $\sw^{-}_{r,s}$ and $\sw^{+}_{r,s}$ map
each set $\wdR(\N^a\E^b)$ into itself.  We will see later
that $\sw^{-}_{r,s}$ maps each set $\wdD_{r,s}(\N^a\E^b)$ into itself.

\subsection{The Irrational-Slope Sweep Map}
\label{subsec:irrat-sweep}

So far, for each rational $-s/r$, we have defined the basic sweep map
$\sw^-_{r,s}$ (which we also call the \textbf{negative-type} sweep
map) and the \textbf{positive-type} sweep map $\sw^+_{r,s}$.  We can
extend this setup to define sweep maps $\sw_{\beta}$ indexed by
\emph{irrational} numbers $\beta$. We regard inputs to $\sw_{\beta}$
as lattice paths $\mkpath(w)\in\patR(\N^a\E^b)$ consisting of
``wands'' with tips at their north ends and east ends.  There is a
``sweep line'' $y-\beta x=k$ that sweeps through the plane as $k$
decreases from just below zero to negative infinity, then from
positive infinity to zero. Because $\beta$ is irrational, every sweep
line intersects at most one wand tip in $\mkpath(w)$.  We obtain
$\sw_{\beta}(w)$ by writing the wands in the order in which the sweep
line hits the wand tips.

For fixed $a,b\geq 0$ and fixed $r>0,s<0$, one readily checks that
$\sw^{-}_{r,s}(w)=\sw_{\beta}(w)$ for all irrationals $\beta$ with
$\beta< -s/r$ and $\beta$ sufficiently close to $-s/r$.  On the other
hand, $\sw_{r,s}^+(w)=\sw_{\beta}(w)$ for all irrationals $\beta$ with
$\beta> -s/r$ and $\beta$ sufficiently close to $-s/r$.  This explains
the terminology ``positive-type'' and ``negative-type'' sweep map.
One approach to studying the sweep map is to understand the
``jump discontinuities'' between $\sw^{+}_{r,s}$ and $\sw^{-}_{r,s}$
that occur at certain critical rationals $-s/r$. 

For irrational $\beta$, we say $w$ is a \textbf{$\beta$-Dyck word}
iff for all lattice points $(x,y)$ visited by $\mkpath(w)$,
$y-\beta x\geq 0$. 

\begin{prop}
(a) For all irrational $\beta$, if $w\in\{\N,\E\}^*$ is a $\beta$-Dyck word
then $v=\sw_{\beta}(w)$ is a $\beta$-Dyck word.
(b) For all $r,s\in\Z$, if $w\in\{\N,\E\}^*$ is an
$(r,s)$-Dyck word, then $v=\sw_{r,s}^-(w)$ is an $(r,s)$-Dyck word.  
\end{prop}
\begin{proof} 
(a) If $\beta<0$ then all words are $\beta$-Dyck words, 
so the proposition certainly holds.  
Therefore, in the following we assume $\beta > 0$.  
Fix $j > 0$; it suffices to show that the point $(x,y)$ reached by
the first $j$ steps of $\mkpath(v)$ satisfies $y>\beta x$. Since
$\beta$ is irrational and $w$ is a $\beta$-Dyck word, there exists a
real $k>0$ such that $v_1\cdots v_j$ consists of all the symbols in
$w$ with wand tips at levels higher than $k$.  In other words,
$v_1\cdots v_j$ is a rearrangement of all the steps of $\mkpath(w)$
that end at points $(x,y)$ with $y>\beta x+k$.  Now, $\mkpath(w)$
begins at the origin, which has level zero.  In general, this path
enters and leaves the region $R_k=\{(x,y): y>\beta x+k\}$ several
times.  Let $w^{(1)},\ldots,w^{(t)}$ be the maximal substrings of
consecutive letters in $w$ such that every step of $w^{(i)}$ ends in
$R_k$.  Every $w^{(i)}$ begins with a north step that enters $R_k$
from below, and every $w^{(i)}$ except possibly $w^{(t)}$ is followed
by an east step that exits $R_k$ to the right. Suppose $w^{(i)}$
consists of $a_i$ north steps and $b_i$ east steps; since the boundary
of $R_k$ is a line of slope $\beta$, the geometric fact in the
previous sentence implies that $a_i/b_i>\beta$ for $1\leq i\leq t$.  By
definition, $v_1\cdots v_j$ is some rearrangement of
$a=a_1+\cdots+a_t$ north steps and $b=b_1+\cdots+b_t$ east steps.  Now
$a_i>\beta b_i$ for all $i$ implies $a>\beta b$. Thus, $v_1\cdots v_j$
is a path from $(0,0)$ to $(b,a)$, where $a>\beta b$.  Since $j$ was
arbitrary, $\mkpath(v)$ is a $\beta$-Dyck path.

(b) It suffices to treat the case $r>0$, $s<0$.  
Fix $a,b\geq 0$ and $w\in\wdD_{r,s}(\N^a\E^b)$.
We can choose an irrational $\beta<-s/r$ so close to $-s/r$
that the region $S=\{(x,y)\in\mathbb{R}^2: \beta x\leq y<(-s/r)x,y\leq a\}$
contains no lattice points, and $v=\sw_{\beta}(w)=\sw_{r,s}^{-}(w)$.
Since $w$ is an $(r,s)$-Dyck word and $\beta<-s/r$, $w$ is also
a $\beta$-Dyck word. By part (a), $v=\sw_{\beta}(w)$ is a
$\beta$-Dyck word. Since $S$ contains no lattice points, $v$
is an $(r,s)$-Dyck word, as needed.

\end{proof}

\subsection{Reversed and Transposed Sweeps}
\label{subsec:minor-var}

For a fixed choice of $r$ and $s$, there are four 
parameters that can be used to define a potential sweep map:
\begin{itemize}
\item the level to start sweeping at,
\item the direction of sweep for a given level 
(i.e., right-to-left or left-to-right),
\item the relative order in which to visit levels (i.e., $k+1$ after level
  $k$ versus $k-1$ after level $k$),
\item the convention for levels assigned to steps (i.e., using the
  west-south convention or the east-north convention).
\end{itemize}
Empirical evidence suggests that for each of the $8=2^3$ possible
choices for the second through fourth parameters, there is a unique
choice of starting level that will lead to a bijective sweep map for
general $\patR(\N^a\E^b)$.  In fact, each of these maps is closely
related to the others through the following two natural involutions on words.
Let $\rev:\{\N,\E\}^*\rightarrow\{\N,\E\}^*$ be the \textbf{reversal map} given
by $\rev(w_1w_2\cdots w_n)=w_n\cdots w_2 w_1$.  Let
$\flip:\{\N,\E\}^*\rightarrow\{\N,\E\}^*$ be the \textbf{transposition map}
that acts by interchanging $\N$'s and $\E$'s.  Evidently, both
$\rev$ and $\flip$ are involutions; $\rev$ maps $\patR(\N^a\E^b)$
bijectively onto itself, whereas $\flip$ maps $\patR(\N^a\E^b)$
bijectively onto $\patR(\N^b\E^a)$.  We can modify the sweep maps by
composing on the left or right with $\rev$ and/or $\flip$. The new
maps are bijections (between appropriate domains and codomains)
iff the original sweep maps are bijections.
Table~\ref{tab:sweeps} displays the eight maps along with their 
relationships to $\sw^{-}_{r,s}$ and $\sw^{+}_{r,s}$.
One can also check that
$$\rev\circ\sw^+_{-r,-s}=\sw^-_{r,s}\mbox{ and }
 \flip\circ\sw^{-}_{s,r}\circ\> \flip=\sw^-_{r,s}.$$

The following degenerate cases of the sweep map are readily verified.  
\begin{itemize}
\item If $r<0$ and $s<0$, then $\sw_{r,s}^{\pm}=\id$, the
 identity map on $\{\N,\E\}^*$.
\item If $r>0$ and $s>0$, then $\sw_{r,s}^{\pm}=\rev$, the reversal map.
\item For $r=s=0$, $\sw_{0,0}^-=\rev$ and $\sw_{0,0}^+=\id$.
\item If $r=0$ and $s<0$, then $\sw_{r,s}^+=\id$,
 whereas $\sw_{r,s}^-$ maps $\N^{a_0}\E\N^{a_1}\E\N^{a_2}\cdots\E\N^{a_k}$
(where $a_j\geq 0$) to $\N^{a_1}\E\N^{a_2}\E\cdots\N^{a_k}\E\N^{a_0}$.
Similar statements hold in the cases: $r=0$ and $s>0$; 
 $s=0$ and $r<0$; $s=0$ and $r>0$.  
\end{itemize}
\begin{table}[ht]
  \centering
  \caption{Symmetries of sweep maps.}
  \begin{tabular}{@{}cccccc@{}} \toprule
    Map & Step-labeling &  Order to & Sweep direction & Start level\\
        & convention  &  scan levels& on each level & \\\midrule
    $\sw^{-}_{r,s}$ & E-N & decreasing & $\leftarrow$ & $-1$\\
    $\sw^{+}_{r,s}$ & E-N & decreasing & $\rightarrow$  & $0$ \\
    $\rev \circ \sw^{-}_{r,s}$ & E-N & increasing & $\rightarrow$  & $0$ \\
    $\rev \circ \sw^{+}_{r,s}$ & E-N & increasing & $\leftarrow$  & $1$ \\
  $\sw^{-}_{r,s}\circ\> \rev$ & W-S & increasing & $\rightarrow$  & $ra+sb+1$ \\
    $\sw^{+}_{r,s}\circ\> \rev$ & W-S & increasing & $\leftarrow$  & $ra+sb$ \\
    $\rev\circ \sw^{-}_{r,s}\circ\> \rev$ & W-S & decreasing & $\leftarrow$ 
    & $ra+sb$ \\
    $\rev\circ \sw^{+}_{r,s}\circ\> \rev$ & W-S & decreasing & $\rightarrow$ 
    & $ra+sb-1$
\\\bottomrule
  \end{tabular}
  \label{tab:sweeps}
\end{table}


\subsection{The General Sweep Map}
\label{subsec:gen}

Suppose $A=\{x_1,\ldots,x_k\}$ is a given alphabet and
$\wt:A\rightarrow\Z$ is a function assigning an integer
\textbf{weight} to each letter in $A$.  Given a word $w=w_1w_2\cdots
w_n\in A^*$, define the \emph{levels $l_0,\ldots,l_n$ relative to
  the weight function $\wt$} by setting $l_0=0$ and, for $1\leq i\leq
n$, letting $l_{i}=l_{i-1}+\wt(w_i)$.  (These levels are
essentially computed according to the east-north convention, though
the west-south convention works equally well.)  Define
$\sw_{\wt}:A^*\rightarrow A^*$ as follows: For each $k$ from $-1$ down
to $-\infty$ and then from $\infty$ down to $0$, scan $w$ from right
to left, writing down each $w_i$ with $l_i=k$ and $i>0$.  Let
$\wdR(x_1^{n_1}\cdots x_k^{n_k})$ be the set of words $w\in A^*$
consisting of $n_j$ copies of $j$ for $1\leq j\leq k$.  Let
$\wdD_{\wt}(x_1^{n_1}\cdots x_k^{n_k})$ be the set of such words for
which all levels $l_i$ are nonnegative.
\begin{conj}\label{conj:gen-sweep}
  Let $A=\{x_1,\ldots,x_k\}$ be an alphabet and $\wt:A\rightarrow\Z$
  a weight function.  For any nonnegative integers $n_1,n_2,\ldots,
  n_k$,
\begin{itemize}
\item[(a)] 
 $\sw_{\wt}$ maps $\wdR(x_1^{n_1}\cdots x_k^{n_k})$ bijectively to itself.
\item[(b)]
 $\sw_{\wt}$ maps $\wdD_{\wt}(x_1^{n_1}\cdots x_k^{n_k})$ bijectively to itself.
\end{itemize}
\end{conj}

\section{Algorithms Equivalent to a Sweep Map}
\label{sec:alg-sweep}

This section reviews some algorithms that have appeared in the literature 
that are equivalent to special cases of the sweep map and its variations.
We describe each algorithm and indicate its exact relation to the general
sweep map. The algorithms reviewed here fall into three main classes:
algorithms that operate on area vectors, algorithms that operate on 
hook-lengths of cells in a partition, and algorithms involving generators
of certain semi-modules. The algorithms based on area vectors arose
in the study of the $q,t$-Catalan polynomials and their generalizations;
these polynomials will be discussed at greater length later in this paper.
The algorithms involving hook-lengths and semi-modules were introduced
to study the special case of Dyck objects where the dimensions $a$ and $b$ are 
coprime. The sweep map provides a single unifying framework that simultaneously
generalizes all these previously studied algorithms. We find it remarkable
that this map, which has such a simple definition, encodes such a rich
array of mathematical structures.

\subsection{Algorithms Based on Area Vectors}
\label{subsec:alg-area-vector}

\subsubsection{Introduction}

This subsection studies several algorithms that operate on lattice
paths by manipulating an \emph{area vector} that records how many area
cells in each row lie between the path and a diagonal boundary. The
simplest version of these algorithms is a bijection on Dyck paths
described in a paper by Haglund and Loehr~\cite[\S3,
Bijections]{HL-park}.
In~\cite{loehr-mcat,loehr-thesis,loehr-trapz}, Loehr generalized
this bijection to define a family of maps $\phi$
acting on $m$-Dyck paths and on lattice paths contained in certain trapezoids.
We begin our discussion with the maps for trapezoidal lattice paths,
which contain the earlier maps as special cases. 
We then look at a generalization of $\phi$ acting on
lattice paths inside squares, followed by a different variation
that acts on Schr\"oder paths containing diagonal steps.

\subsubsection{Trapezoidal Lattice Paths}
\label{subsubsec:trapz}

Fix integers $k\geq 0$ and $n,m>0$.  Let $T_{n,k,m}$ denote
the set of \emph{trapezoidal lattice paths} from $(0,0)$ to $(k+mn,n)$ that
never go strictly to the right of the line $x=k+my$.  
The paper~\cite{loehr-trapz} introduces a bijection
$\phi=\phi_{n,k,m}:T_{n,k,m}\rightarrow T_{n,k,m}$ 
and its inverse. 
That paper (last paragraph of Section 3.1)
accidentally switches the roles of $\phi$ and $\phi^{-1}$
compared to~\cite{loehr-mcat} and other literature. The map $\phi_{n,k,m}$
discussed below is the composite $\alpha^{-1}\circ\beta\circ\gamma$
from~\cite{loehr-trapz} (which is erroneously 
denoted $\phi^{-1}$ in that paper). After recalling the
definition of this map, we show that 
a variant of $\phi_{n,k,m}$ is a sweep map.

Given a path $P\in T_{n,k,m}$, we first construct the
\textbf{
area vector} $g(P)=(g_0,g_1,\ldots,g_{n-1})$,
where $g_i$ is the number of complete lattice squares in the
horizontal strip $\{(x,y): x\geq 0, i\leq y\leq i+1\}$ that lie to the
right of $P$ and to the left of the line $x=k+my$.  The area vector
$g(P)$ has the following properties: $0\leq g_0\leq k$; $g_i$ is a
nonnegative integer for $0\leq i<n$; and $g_i\leq g_{i-1}+m$ for
$1\leq i<n$. One readily checks that $P\mapsto g(P)$ is a bijection
from $T_{n,k,m}$ to the set of vectors of length $n$ with the
properties just stated.

For $P\in T_{n,k,m}$, we compute $\phi(P)$ by concatenating
lattice paths (regarded as words in $\{\N,\E\}^*$) that are built up
from subwords of $g(P)$ as follows. For $i=0,1,2,\ldots$, let
$z^{(i)}$ be the subword of $g(P)$ consisting of symbols in the
set $\{i,i-1,i-2,\ldots,i-m\}$; let $M$ be the largest $i$ such that $z^{(i)}$
is nonempty.  Create a word $\sigma^{(i)}\in\{\N,\E\}^*$
from $z^{(i)}$ by replacing each symbol $i$ in $z^{(i)}$ by $\N$ and
replacing all other symbols in $z^{(i)}$ by $\E$.  
Let $\sigma$ be the concatenation of words
\[ \sigma=\sigma^{(0)}\,\,\E\sigma^{(1)}\,\,\E\sigma^{(2)}\cdots\,\,
\E\sigma^{(k)}\, \sigma^{(k+1)}\,\cdots\,\sigma^{(M)}, \]
in which an extra east step is added after the first $k$ words.
Define $\phi(P)=\mkpath(\sigma)$.
It is proved in~\cite[Sec. 3]{loehr-trapz} that $\phi(P)$
always lies in $T_{n,k,m}$, and that $\phi_{n,k,m}$ is a bijection.

To relate $\phi$ to the sweep map, we need to introduce
a modified map $\phi'$ that incorporates the
bijection described in~\cite[Sec. 4]{loehr-trapz}.
Keep the notation of the previous paragraph.  
For all $i$ with $k<i\leq M$, note that $\sigma^{(i)}$ must begin
with an $\E$, so we can write $\sigma^{(i)}=\E\tilde{\sigma}^{(i)}$. 
Let $\tau^{(i)}=\rev(\sigma^{(i)})$ for $0\leq i\leq k$,
and let $\tau^{(i)}=\E\,\rev(\tilde{\sigma}^{(i)})$ for $k<i\leq M$.
Define $\phi'(P)=\mkpath(\tau)$, where
\[ \tau=\tau^{(0)}\,\,\E\tau^{(1)}\,\,\E\tau^{(2)}\cdots\,\,\E\tau^{(k)}
\,\tau^{(k+1)}\,\cdots\,\tau^{(M)}. \]
\begin{example}\label{ex:phi}
Let $n=8$, $k=2$, $m=2$, and $\mkwd(P)=\NE{ENNEENEEEEENNEEENNEEENEEEE}$.
Then $g(P)=(1,3,3,0,2,1,3,2)$, so 
\[ \begin{array}{llllll}
z^{(0)}=0, & z^{(1)}=101, & z^{(2)}=10212, & z^{(3)}=1332132,
& z^{(4)}=33232,& z^{(5)}=333, \\
\sigma^{(0)}=\NE{N}, & \sigma^{(1)}=\NE{NEN}, & \sigma^{(2)}=\NE{EENEN}, &
\sigma^{(3)}=\NE{ENNEENE}, & \sigma^{(4)}=\NE{EEEEE}, & \sigma^{(5)}=\NE{EEE},\\
\tau^{(0)}=\NE{N}, & \tau^{(1)}=\NE{NEN}, & \tau^{(2)}=\NE{NENEE}, &
\tau^{(3)}=\NE{EENEENN}, & \tau^{(4)}=\NE{EEEEE}, & \tau^{(5)}=\NE{EEE},
\end{array}\] 
\begin{align*}
 \mkwd(\phi(P)) &= \sigma = 
   \NE{N\,\,ENEN\,\,EEENEN\,\,ENNEENE\,\,EEEEE\,\,EEE}, \\
 \mkwd(\phi'(P)) &= \tau = 
   \NE{N\,\,ENEN\,\,ENENEE\,\,EENEENN\,\,EEEEE\,\,EEE}.
\end{align*}
\end{example}

\begin{theorem}\label{thm:trapz-vs-sweep}
 For all $k\geq 0$, $n,m>0$, and $P\in T_{n,k,m}$, 
$$\mkwd(\phi'_{n,k,m}(P))=
\flip\circ\rev\circ\sw^{-}_{1,-m}\circ\>\rev\circ\flip(\mkwd(P)).$$
\end{theorem} 
\begin{proof}
\noindent\textbf{Step 1.} Write $w=\mkwd(P)=w_1w_2\cdots w_{k+n+nm}$ and
$\sw'=\flip\circ\rev\circ\sw^{-}_{1,-m}\circ\>\rev\circ\flip$.  One
may routinely check that $\sw'(w)$ may be computed by the following
algorithm.  Let $l_1=k$, $l_{j+1}=l_{j}+m$ if $w_j=\N$, and
$l_{j+1}=l_{j}-1$ if $w_j=\E$. (Thus in this variation, a north step
$w_j$ from $(x,y)$ to $(x,y+1)$ has associated level $l_j=k+my-x\geq
0$, whereas an east step $w_j$ from $(x,y)$ to $(x+1,y)$ has
associated level $l_j=k+my-x>0$.  
Up to the shift by $k$, this is the
west-south convention for assigning levels.) 
Generate an output word
$y$ from left to right as follows.  For each level $L=0,1,2,\ldots$,
scan $w$ from right to left, and append the letter $w_j$ to the right
end of $y$ whenever an index $j\leq k+n+nm$ is scanned for which
$l_{j}=L$. For each $i\geq 0$, let $\rho^{(i)}$ be the subword of $y$
generated in the $L=i$ iteration of the algorithm.

  In the preceding example, the sequence of steps and levels is

\begin{center}
{\setlength{\tabcolsep}{1pt}
\begin{tabular}{ccccccccccccccccccccccccccc} 
E&N&N&E&E&N&E&E&E&E&E&N&N&E&E&E&N&N&E&E&E&N&E&E&E&E  &
\\2&1&3&5&4&3&5&4&3&2&1&0&2&4&3&2&1&3&5&4&3&2&4&3&2&1 &(0),
\end{tabular}}
\end{center}

\noindent
where the zero in parentheses is the level following the final east step.  
The subwords $\rho^{(i)}$ are
  \[ \rho^{(0)}=\NE{N},\ \rho^{(1)}=\NE{ENEN},\ \rho^{(2)}=\NE{ENENEE},\ 
  \rho^{(3)}=\NE{EENEENN},\ \rho^{(4)}=\NE{EEEEE},\ \rho^{(5)}=\NE{EEE}. \] 
By definition of the levels, one sees that the maximum level of any letter
in $w$ is the same value $M$ appearing in the definition of $\phi'(P)$.
Since $y=\sw'(w)=\rho^{(0)}\rho^{(1)}\cdots\rho^{(M)}$,
it will suffice to prove that $\rho^{(0)}=\tau^{(0)}$,
$\rho^{(i)}=\E\tau^{(i)}$ for $1\leq i\leq k$, and $\rho^{(i)}=\tau^{(i)}$ 
for all $i$ with $k<i\leq M$ (as illustrated by Example~\ref{ex:phi}). 
Define $\tilde{\rho}^{(i)}=\rev(\rho^{(i)})$ for all $i$; 
this is the word obtained 
by scanning $w$ from left to right and taking all letters at level $i$.
By reversing everything, it is enough to prove that
$\tilde{\rho}^{(0)}=\sigma^{(0)}$, $\tilde{\rho}^{(i)}=\sigma^{(i)}\E$ 
for $1\leq i\leq k$,
and $\tilde{\rho}^{(i)}=\tilde{\sigma}^{(i)}\E$ for all $i$ with $k<i\leq M$.

\medskip\noindent\textbf{Step 2.} 
  Fix a level $L=i$.  We define an \textbf{event
  sequence} in $\{\mathrm{A},\mathrm{B},\mathrm{C}\}^*$ associated
  with a left-to-right scan of level $i$.  It follows from Step 1 that
  the levels of the north steps of $P$ will be
  $g_0,g_1,\ldots,g_{n-1}$ in this order.  As we scan $w$ during this
  iteration, the following \textbf{events} may occur:
  \begin{itemize}
  \item[A.] We scan an $\N$ of $w$ at level $i$,
    which appends an $\N$ onto both $\sigma^{(i)}$ and $\tilde{\rho}^{(i)}$.
  \item[B.] We scan an $\E$ of $w$ at level $i$, which appends an $\E$
    onto $\tilde{\rho}^{(i)}$.
  \item[C.] We scan an $\N$ of $w$ with level in $\{i-1,i-2,\ldots,i-m\}$,
    which appends an $\E$ onto $\sigma^{(i)}$.
  \end{itemize}
  Consider the sequence of events A, B, C that occur during the $L=i$
  scan.  In our example, the $L=2$ scan has event sequence BCBCABCAB,
  whereas the $L=3$ scan has event sequence CAABCBCABCB.

\medskip\noindent\textbf{Step 3.} 
  We prove that $\tilde{\rho}^{(0)}=\sigma^{(0)}$.  
  For the $L=0$ scan,
  events B and C are impossible, since the path stays within the
  trapezoid.  So the event sequence consists of $j$ A's for some $j$, 
  and $\tilde{\rho}^{(0)}$ and $\sigma^{(0)}$ both consist of $j$ $N$'s.

\medskip\noindent\textbf{Step 4.} 
  For $0<i\leq M$, we analyze the possible transitions
  between events A, B, and C that may occur during the $L=i$
  scan. Note that events A and B can only occur when the level of the
  current character in $w$ is $\geq i$, whereas event C only occurs
  when this level is $<i$. Moreover, the only way to transition from a
  level $\geq i$ to a level $<i$ is via event B, and the only way to
  transition from a level $<i$ to a level $\geq i$ is via event C.
  Consequently, in the event sequence for $L=i$, every A (not at the
  end) can only be followed by A or B; every B (not at the end) can
  only be followed by C; and every C (not at the end) can only be
  followed by A or B.  The path $P$ ends at level $0<i$, so the event
  sequence must end in a B.

\medskip\noindent\textbf{Step 5.} 
  We prove that $\tilde{\rho}^{(i)}=\sigma^{(i)} \E$ for $1\leq i\leq k$. 
  Since $i\leq k$ and the origin has level $k$, the first letter
  in the event sequence must be A or B.  By Step 4, the event sequence
  is some rearrangement of A's and BC's, except there is an unmatched
  B at the end.  By definition of the events in Step 2, this means
  that $\tilde{\rho}^{(i)}$ and $\sigma^{(i)}$ agree, except for an 
  extra $\E$ at the end of $\tilde{\rho}^{(i)}$.

\medskip\noindent\textbf{Step 6.} 
 We prove that $\tilde{\rho}^{(i)}=\tilde{\sigma}^{(i)} \E$ 
  for $k<i\leq M$. Since $i>k$ and the origin has level $k$, the first letter 
  in the event sequence must be an unmatched C.  Thereafter, the event
  sequence consists of A's and matched BC pairs, with one unmatched B
  at the end. The initial C gives the initial $\E$ in $\sigma^{(i)}$ that is
  deleted to form $\tilde{\sigma}^{(i)}$. As in Step 5, we see that 
  $\tilde{\rho}^{(i)}$ and $\tilde{\sigma}^{(i)}$ agree, 
  except for an extra $\E$ at the end of $\tilde{\rho}^{(i)}$
  caused by the unmatched B.  
\end{proof}

The proof structure above can be readily adapted to show that other algorithms
based on area vectors are equivalent to suitable sweep maps. For this reason,
we will omit the details of these proofs in the remainder of this subsection.  
For instance, one can modify the preceding proof to show that the map 
$\phi_{n,0,1}$ 
(which acts on Dyck paths of order $n$) is also a sweep map.
\begin{theorem}
For all $n>0$ and $P\in T_{n,0,1}$, 
$\mkwd(\phi_{n,0,1}(P))= \flip\circ\rev\circ\sw^{-}_{1,-1}(\mkwd(P))$. 
\end{theorem}

Similarly, let $\phi_{\HL}$ denote the map
described in~\cite[\S3, Bijections]{HL-park} that sends
unlabeled Dyck paths to unlabeled Dyck paths.

\begin{theorem}
For all $n>0$ and $P\in T_{n,0,1}$, 
$\mkwd(\phi_{\HL}(P))= \sw^{-}_{1,-1}(\mkwd(P))$. 
\end{theorem}

We note that the partition $\mkptn(\phi_{n,0,1}(P))$ is
the transpose of the partition $\mkptn(\phi_{\HL}(P))$. Since
\[ \flip\circ\rev(\mkwd(\pi))=\mkwd(\pi') \]
for all partitions $\pi$ (where $\pi'$ denotes the transpose of $\pi$), 
the theorem for $\phi_{\HL}$
follows from the theorem for $\phi_{n,0,1}$ and vice versa.

\subsubsection{Square Lattice Paths}

In~\cite{LW-square}, Loehr and Warrington modified the map
$\phi_{\HL}$ to obtain a bijection $\phi_{\LW}$ on $\patR(\N^n\E^n)$,  
the set of lattice paths in an $n\times n$ square.
Given $P\in\patR(\N^n\E^n)$, we define its \textbf{area vector}
$g(P)=(g_0,g_1,\ldots,g_{n-1})$ by letting $g_i+n-i$ be the number
of complete squares in the strip $\{(x,y): x\geq 0, i\leq y\leq i+1\}$
that lie to the right of $P$ and to the left of $x=n$.  (This
reduces to the previous area vector if $P$ is a Dyck path.) 
The area vectors of paths in $\patR(\N^n\E^n)$ are
characterized by the following properties: $g_0\leq 0$; $g_i+n-i\geq
0$ for $0\leq i<n$; and $g_i\leq g_{i-1}+1$ for $1\leq i<n$.

Given $P\in\patR(\N^n\E^n)$, we define a new path $\phi_{\LW}(P)$
as follows. For all $i\in\Z$, let $z^{(i)}$ be the subword
of $g(P)$ consisting of all occurrences of $i$ and $i-1$.
Create words $\sigma^{(i)}$ from $z^{(i)}$ by replacing each $i$
by $\E$ and each $i-1$ by $\N$.
For all $i\geq 0$, let $\tau^{(i)}$ be the reversal of $\sigma^{(i)}$.  
For all $i<0$, $\sigma^{(i)}$ must end in $\E$,
so we can write $\sigma^{(i)}=\tilde{\sigma}^{(i)}\E$; 
let $\tau^{(i)}=\rev(\tilde{\sigma}^{(i)})\E$. Finally, define 
\[ \tau=\tau^{(-1)}\tau^{(-2)}\cdots \tau^{(-n)}
   \tau^{(n)}\cdots \tau^{(2)}\tau^{(1)}\tau^{(0)}, \]
and set $\phi_{\LW}(P)=\mkpath(\tau)$.

\begin{example}
Let $P\in\patR(\N^{16}\E^{16})$ be such that 
\[\mkwd(P)=w=\NE{ENEENENNNNEENEEEENNNENEENNNNENEE}.\]
Then $g(P)=(-1,-2,-2,-1,0,1,0,-3,-2,-1,-1,-2,-1,0,1,1)$, so (for
instance) $z^{(1)}=010011$, $\sigma^{(1)}=\NE{NENNEE}$, 
$\tau^{(1)}=\NE{EENNEN}$, 
$z^{(-2)}=\text{$-2$ $-2$ $-3$ $-2$ $-2$}$, 
$\sigma^{(-2)}=\NE{EENEE}$, $\tau^{(-2)}=\NE{ENEEE}$,
and so on.  Next, $\tau$ is the concatenation of
the words $\tau^{(-1)}=\NE{NEENENNEE}$, $\tau^{(-2)}=\NE{ENEEE}$,
$\tau^{(-3)}=\NE{E}$, $\tau^{(2)}=\NE{NNN}$, $\tau^{(1)}=\NE{EENNEN}$,
and $\tau^{(0)}=\NE{ENNNEENN}$, and $\phi_{\LW}(P)=\mkpath(\tau)$.
On the other hand, the reader can
check that $\tau=\sw^{-}_{1,-1}(w)$.  In fact, for all $i\in\Z$,
the subword $\tau^{(i)}$ is precisely the subword $\rho^{(i)}$
of letters in $\sw^{-}_{1,-1}(w)$ 
coming from letters in $w$ with $(1,-1)$-level (using the E-N
convention) equal to $i$.  One can prove this always happens, by
adapting the ideas in the proof of Theorem~\ref{thm:trapz-vs-sweep},
to obtain the following theorem.
\end{example}

\begin{theorem} 
For all $n>0$ and all $P\in\patR(\N^n\E^n)$, 
$\mkwd(\phi_{\LW}(P))=\sw^{-}_{1,-1}(\mkwd(P))$.  
\end{theorem} 

\subsubsection{Schr\"oder Lattice Paths}

A \textbf{Schr\"oder path} of order $n$ is a path from the
origin $(0,0)$ to $(n,n)$, never going below $y=x$, with the
allowed steps being north steps of length $1$, east steps
of length $1$, and a northeast step of length $\sqrt{2}$.
In~\cite[Theorem 6]{EHKK}, Egge, Haglund, Killpatrick, and Kremer
extend $\phi_{\HL}$ to a bijection $\phi_{\EHKK}$ acting
on Schr\"oder paths.

\begin{theorem} 
 After converting from paths to words, $\phi_{\EHKK}$ is
 the sweep map $\sw_{\wt}$ associated to the alphabet $A=\{\N,\D,\E\}$
 with weight function $\wt(\N) = 1$, $\wt(\D) = 0$, and $\wt(\E) = -1$.
\end{theorem}

\subsection{An Algorithm Based on Hook-Lengths}
\label{subsec:zetamap}

Throughout this subsection, fix positive integers $a$ and $b$ with 
$\gcd(a,b)=1$.  In~\cite{AHJ}, D.~Armstrong, C.~Hanusa and B.~Jones 
investigate the combinatorics of $(a,b)$-cores.  In the process, they define 
a map $\zetamap: \ptnD_{b,-a}(a,b) \rightarrow \ptnD_{b,-a}(a,b)$. 
For $\pi\in\ptnD_{b,-a}(a,b)$, the partition
$\zetamap(\pi)$ is defined in two stages.  First, create a partition
$\nu = \drewnu(\pi)$ as follows.  Consider the levels of all lattice
squares lying above $by=ax$ and below the path $\mkpath(\pi)$. Since
$\gcd(a,b)=1$, these levels must all be distinct.
Sort these levels into increasing order, and write 
them in a column from bottom to top. Let $\nu = \drewnu(\pi)$ 
be the unique partition such that these levels are the hook-lengths of
the cells in the first column. (Recall that the \textbf{hook} of a cell $c$ in
a partition diagram consists of $c$ itself, all cells below $c$ 
in its column, and all cells right of $c$ in its row. The \textbf{hook-length}
of $c$, denoted $h(c)$, is the number of cells in the hook of $c$.)

The second stage maps $\nu$ to a new partition $\rho = \drewrho(\nu)$
as follows.  There will be one nonzero row of $\rho$ for each row of $\nu$
whose first-column hook-length is the level of a square directly east of
a north step of $\mkpath(\pi)$.  
To determine the length of each row in $\rho$, count the number of cells
of $\nu$ in the corresponding row whose hook-length is less than or
equal to $b$.  The $\zetamap$ map is then defined by $\zetamap(\pi) =
\drewrho \circ \drewnu(\pi)$. See Figure~\ref{fig:drewmap-ex} for
an example.

\begin{figure}[htbp]
  \centering 
      {\scalebox{.3}{\includegraphics{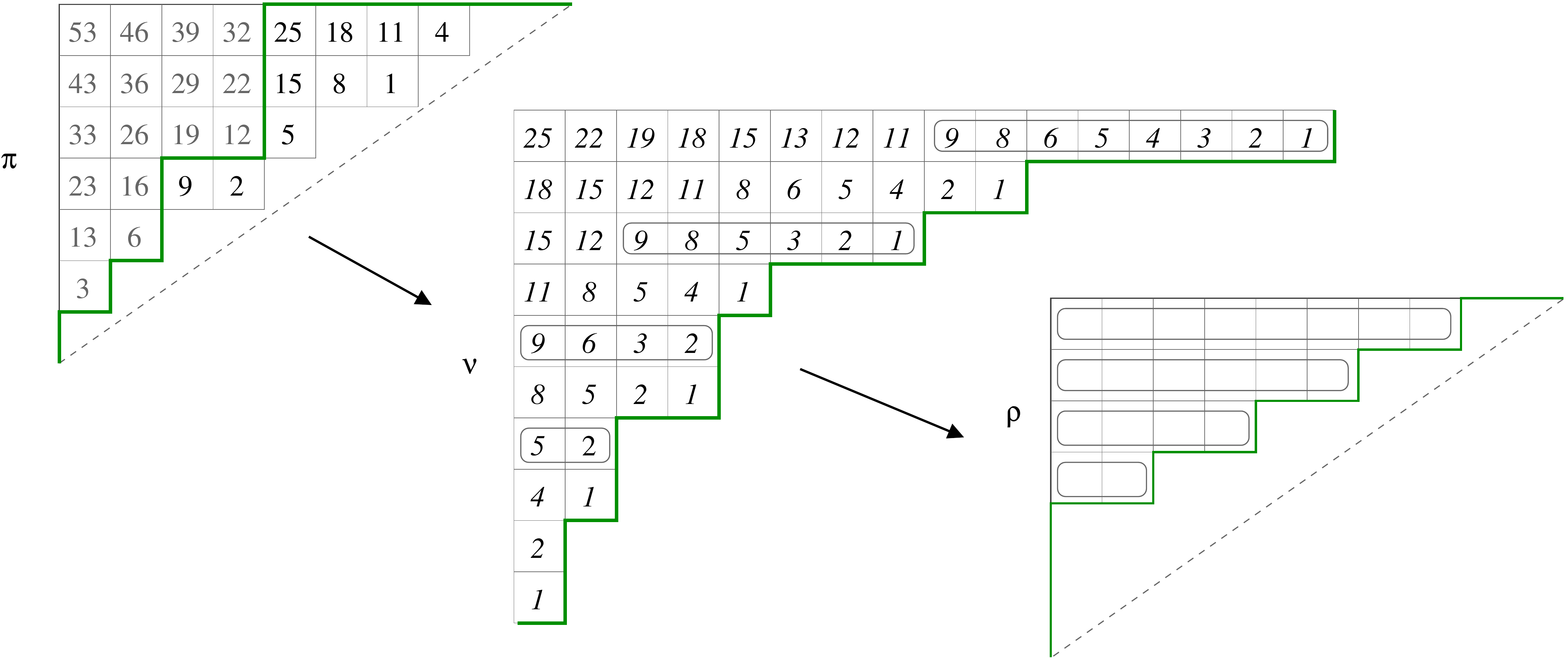}}}
      \caption{Example of the $\zetamap$ map applied to $\pi = (4,4,4,2,2,1)$
        for $a=7$ and $b=10$.  The partitions $\nu=\drewnu(\pi)$
        and $\rho=\drewrho(\nu)=\zetamap(\pi)=(8,6,4,2)$ are shown.  
        Each cell of $\nu$ is labeled with its hook-length.} 
  \label{fig:drewmap-ex}
\end{figure}

\begin{theorem}\label{thm:drew-sweep} 
  For all $\pi\in \ptnD_{b,-a}(a,b)$, 
  $\mkwd(\zetamap(\pi)) = \sw^{+}_{b,-a}\circ\rev(\mkwd(\pi))$.
\end{theorem}

To prove this theorem, we will introduce alternate formulations
of the maps $\drewnu$ and $\drewrho$, denoted $\newnu$ and $\newrho$,
that focus attention on the lattice paths making up the frontiers
of $\pi$, $\nu$, and $\rho$. This will enable us to compare the
action of the $\zetamap$ map on partitions to the action of the sweep map
on lattice paths.  
To start, define $\newnu: \ptnD \rightarrow \ptnR$ by setting $\newnu(\pi) =
\mkptn(z_0z_1z_2\cdots)$, where $z_0=E$, and for all $i > 0$, 
\begin{equation*}
  z_i = 
  \begin{cases}
    \N, & \text{if the square with level $i$ lies between $\mkpath(\pi)$ 
      and the line $by=ax$},\\
    \E, & \text{otherwise.}
  \end{cases}
\end{equation*}
Since $\mkpath(\pi)$ begins at $(0,0)$ and ends at $(b,a)$, we must
have $z_i = \E$ for all $i > ab-a-b$.

\begin{lemma}\label{lem:nu}
  For all $\pi \in \ptnD$, $\newnu(\pi) = \drewnu(\pi)$.
\end{lemma}
\begin{proof}
  We first observe that the partitions $\newnu(\pi)$ and
  $\drewnu(\pi)$ will have the same number of (positive-length) rows.
 For, on one hand,
  the first-column hook-lengths in $\drewnu(\pi)$ will be the levels
  of the squares between $\mkpath(\pi)$ and $by=ax$.  
 On the other hand, these same levels
  will be the indices of the north steps of $\mkpath(\newnu(\pi))$.

  Now we show that the row lengths will be the same.  Consider the
  $i$-th row from the bottom (starting with $i=1$ for the bottom row)
  in each partition.  Suppose the first-column hook-length in this row
  in $\drewnu(\pi)$ is $k$.  By definition, the length of this row in
  $\drewnu(\pi)$ will be $k-i+1$.  Additionally, since we are in the
  $i$-th row from the bottom, it follows that of the values
  $\{0,1,2,\ldots,k-1\}$, exactly $i-1$ are levels of lattice squares
  below $\mkpath(\pi)$, while the remaining $k-i+1$ values are levels of lattice
  squares above $\mkpath(\pi)$.  But the levels corresponding to squares above
  $\mkpath(\pi)$ map to $\E$'s under $\newnu$.  So 
  $|\{j:\,0\leq j\leq k\mbox{ and } z_j = \E\}| = k-i+1$, 
  which implies that the number of cells in the $i$-th
  row from the bottom of $\newnu(\pi)$ will also be $k-i+1$.
\end{proof}

We now define an analog of the map $\drewrho$ that maps $\nu =
\newnu(\pi)=\drewnu(\pi)$ to a new partition $\rho$.  With the word
$z=z_0z_1z_2\cdots$ defined as above, let $y=y_0y_1y_2\cdots$ be the
subword of $z$ formed by retaining only those $z_i$ for which $i+a$ is
the level (using the W-S convention) of a step of $\mkpath(\pi)$; then
set $\newrho(\nu)=\mkptn(y)$.  
(Technically, $\newrho$ depends not only on the partition $\nu$, but
on $a$, $b$ and $\pi = \newnu^{-1}(\nu)$ as well.  However, $a$ and $b$
are fixed and we consider $\newrho$ only as part of the composition
$\newrho\circ \newnu$.)
  
\begin{lemma}\label{lem:drew-sweep}
For all $\pi\in\ptnD_{b,-a}(a,b)$,
$\mkwd(\newrho\circ\newnu(\pi)) = \sw^{+}_{b,-a}\circ \rev(\mkwd(\pi))$.
\end{lemma}
\begin{proof}
  Recall from Table~\ref{tab:sweeps} that the sweep map variation on
  the right side of the lemma acts on $\mkwd(\pi)$ by scanning the
  levels $0,1,2,\ldots$ in this order, sweeping up path steps with
  levels assigned according to the W-S convention. (Since
  $\gcd(a,b)=1$, each level appears at most once in $\mkwd(\pi)$.
  Also, for $r=b$ and $s=-a$, $ra + sb = 0$.)

 To compare this sweep map to the action of $\newrho\circ\newnu$,
 first note that a north step on the frontier of $\pi$ has level $i+a$
 iff the lattice square directly east of that north step has level $i$.
 Such lattice squares are encoded as $z_i=\N$ by $\newnu$ and
 are then retained by $\newrho$.
 Similarly, an east step on the frontier of $\pi$ has level $i+a$
 iff the lattice square directly north of that east step has level $i$.
 Such lattice squares are encoded as $z_i=\E$ by $\newnu$ and
 are then retained by $\newrho$. All other lattice squares not of
 these two types are discarded by $\newrho$. Thus 
 $\mkwd(\newrho\circ\newnu(\pi))$,
 which is precisely the subword of $z_0z_1z_2\cdots$ 
 consisting of letters retained by $\newrho$, will be the same 
 word produced by the sweep map.  
\end{proof}

 Intuitively, one can think of $\newnu$ as sweeping up \emph{all} 
 lattice-square levels, and then $\newrho$ keeps only those levels
 of squares that are ``adjacent'' to the frontier of $\pi$ in the
 sense described above.  Below, we will call these squares \emph{frontier
 squares} of $\pi$.

 Before proving our final lemma, we need to introduce some temporary 
 notation for describing the cells and rows of $\nu=\drewnu(\pi)=\newnu(\pi)$.
 First, let $\FCHL_{\nu}$ be the set of hook-lengths of cells
 in the first (leftmost) column of $\nu$. In our running example,
 $\FCHL_{\nu}=\{1,2,4,5,8,9,11,15,18,25\}$. Each square $c$ in the
 diagram of $\nu$ lies due north of an east step on the frontier of
 $\nu$, say $z_i=\E$; and $c$ lies due west of a north step on the
 frontier of $\nu$, say $z_m=\N$. Identify the square $c$ with
 the ordered pair of labels $[i,m]$. Observe that $[i,m]$ is the
 label of some cell $c$ in the diagram of $\nu$ iff $0\leq i<m$
 and $z_i=\E$ and $z_m=\N$; in this case, we must have $m\in\FCHL_{\nu}$.
 It is routine to check that the hook-length $h(c)$ is $m-i$.
 For all $m\in\FCHL_{\nu}$, the \emph{row of $\nu$ indexed by $m$}
 is the row with leftmost cell $[0,m]$, whose hook-length is $m$.

\begin{lemma}\label{lem:rho}
  For all $\pi \in \ptnD_{b,-a}(a,b)$, 
 $\newrho(\newnu(\pi)) = \drewrho(\drewnu(\pi))$.  
\end{lemma}
\begin{proof}
  Let $\nu=\newnu(\pi)=\drewnu(\pi)$. 
  We must show $\drewrho(\nu)=\newrho(\nu)$.

  \medskip\noindent\textbf{Step 1.}  We show that $\drewrho$ and
  $\newrho$ keep the same rows of $\nu$.  On one hand, the definition
  of $\drewrho$ tells us to keep the rows of $\nu$ indexed by those
  $m\in\FCHL_{\nu}$ appearing as the level of a square immediately east of
  a north step in $\mkpath(\pi)$. These squares are the frontier
  squares of $\pi$ below $\mkpath(\pi)$. On the other hand, the
  definition of $\newrho$ tells us to retain the frontier steps
  $z_m=\N$ of $\nu$ for those $m\in\FCHL_{\nu}$ such that $m+a$ is the
  level of a north step of $\mkpath(\pi)$. As observed in the earlier
  lemma, these $m$'s correspond to $m$'s that are the levels of
  frontier squares of $\pi$ below $\mkpath(\pi)$.  So $\drewrho$ and
  $\newrho$ do retain the same rows of $\nu$.

\medskip\noindent\textbf{Step 2.} 
 For each fixed $m\in\FCHL_{\nu}$, we compare the cells in the row
 of $\nu$ indexed by $m$ that are discarded by $\drewrho$ and $\newrho$.  
On one hand, let \[ C_m = \{c:\,\text{$c$ is a cell
 in the row of $\nu$ indexed by $m$, and $h(c)>b$}\}. \]
The cells in $C_m$ are erased by $\drewrho$. So, for those row indices $m$ 
retained by $\drewrho$, $|C_m|$ is the difference between the length of this
row in $\nu$ and the length of the corresponding row in $\drewrho(\nu)$.
On the other hand, let \[ B_m = \{j:\,1\leq j < m, z_j = \E,\text{ and 
 the square with level $j$ is not a frontier square of $\pi$}\}. \]
The values $j\in B_m$ index the east steps $z_j=\E$ prior to the
north step $z_j=m$ that are discarded by $\newrho$. So, for those row indices
$m$ retained by $\newrho$, $|B_m|$ is the difference between the length
of this row in $\nu$ and the length of the corresponding row in
$\newrho(\nu)$. Since $\drewrho$ and $\newrho$ retain the same row
indices $m$ (by Step 1),
it will now suffice to show that $|B_m|=|C_m|$ for all $m\in\FCHL_{\nu}$.
  
\medskip\noindent\textbf{Step 3.} 
Fix $m\in\FCHL_{\nu}$; we define bijections $G:B_m\rightarrow C_m$
and $H:C_m\rightarrow B_m$. Given $j\in B_m$, let $G(j)=[j-b,m]$.
Given a cell $[i,m]\in C_m$, let $H([i,m])=i+b$. It is clear that
$H\circ G$ and $G\circ H$ are identity maps, so the proof will be
complete once we check that $G$ does map $B_m$ into $C_m$, and
$H$ does map $C_m$ into $B_m$. Consider a fixed $j\in B_m$. Since
$z_j=\E$, $j$ is the level of a square above $\mkpath(\pi)$, but this square is
not a frontier square of $\pi$. Hence, the square directly below this
square (whose level is $j-b$) is also above $\mkpath(\pi)$. This
implies $j-b\geq 0$ and $z_{j-b}=\E$. Moreover, since $j<m$, the hook-length
of the cell $[j-b,m]$ is $m-(j-b)=b+(m-j)>b$, proving that
$G(j)=[j-b,m]\in C_m$. Now consider a fixed cell $[i,m]\in C_m$.
By definition of $C_m$, we must have $z_i=\E$ and $m-i>b$.
So the square with level $i$ is above $\mkpath(\pi)$, and hence the
square with level $i+b$ is also above $\mkpath(\pi)$ and is not a frontier
square of $\pi$. In particular, $z_{i+b}=\E$.
Finally, $i+b<m$ and $i+b>0$, so $i+b\in B_m$. 
We conclude that $H([i,m])\in B_m$.  
\end{proof}
 
In our running example, the row indices $m$ retained by both $\drewrho$ 
and $\newrho$ are $5,9,15,25$. For $m=25$, we have
\[ B_{25}=\{10,13,16,17,20,22,23,24\}; \]
\[ C_{25}=\{[0,25],[3,25],[6,25],[7,25],[10,25],[12,25],[13,25],[14,25]\}. \]
($C_{25}$ is the set of the leftmost eight cells in the top row of $\nu$.)
The map $j\mapsto [j-10,25]$ defines a bijection from $B_{25}$ to $C_{25}$.

\subsection{An Algorithm Based on Semi-Module Generators}
\label{subsec:gorsky-mazin}

In~\cite{GM-jacI,GM-jacII}, E.~Gorsky and M.~Mazin relate the
$q,t$-Catalan numbers and their generalizations to the homology of
compactified Jacobians for singular plane curves with Puiseux pair
$(a,b)$.  In the course of their investigations, they introduce the
following map $\gmmap_{b,a}$ on partitions in $\ptnD_{b,-a}(a,b)$ (we
follow the notation of ~\cite{GM-jacII}).  Let $a,b\in\Z_{>0}$ with
$\gcd(a,b)=1$ and $\pi\in\ptnD_{b,-a}(a,b)$.  For $1\leq i\leq b$,
define the \textbf{$\boldsymbol{b}$-generators of $\boldsymbol{\pi}$},
denoted $\beta_1 < \cdots < \beta_b$, to be the levels of the squares
immediately above $\mkpath(\pi)$.  Define $\Delta = \Delta(\pi)$ to be
the set of levels of \emph{all} lattice squares lying north or west of
$\mkpath(\pi)$ (i.e., including squares not adjacent to
$\mkpath(\pi)$).  Equivalently, $\Delta=\Z_{\geq 0}\setminus\Delta^c$
where $\Delta^c$ is the finite set of levels of lattice squares
between $\mkpath(\pi)$ and $by=ax$.  We then define a new partition
$\rho = \gmmap_{b,a}(\pi)$ by setting the $i$-th column of $\rho$ to have
length
\begin{equation*}
  g_{b,a}(\beta_i) = |\{\beta_i,\beta_i+1,\ldots,\beta_i+a-1\} \setminus \Delta|
               = |\{\beta_i,\beta_i+1,\ldots,\beta_i+a-1\}\cap\Delta^c|.
\end{equation*}

For our running example where $\pi = (4,4,4,2,2,1)$, $a=7$, and
$b=10$,  the $10$-generators are $\{0,3,6,7,12,14,19,21,28,35\}$, 
\begin{equation*} 
\Delta^c = \{1,2,4,5,8,9,11,15,18,25\}, 
\end{equation*}
and $\Delta=\mathbb{Z}_{\geq 0}\setminus\Delta^c$.  It follows that 
\begin{align*}
  g_{10,7}(0) &= |\{0,\ldots,6\} \setminus \{0,3,6\}\}| = 4,\\
  g_{10,7}(3) &= |\{3,\ldots,9\} \setminus \{3,6,7\}\}| = 4,\\
  g_{10,7}(6) &= |\{6,\ldots,12\} \setminus \{6,7,10,12\}\}| = 3,\\
  g_{10,7}(7) &= |\{7,\ldots,13\} \setminus \{7,10,12,13\}\}| = 3,\\
  g_{10,7}(12) &= |\{12,\ldots,18\} \setminus \{12,13,14,16,17\}\}| = 2,\\
  g_{10,7}(14) &= |\{14,\ldots,20\} \setminus \{14,16,17,19,20\}\}| = 2,\\
  g_{10,7}(19) &= |\{19,\ldots,25\} \setminus \{19,20,21,22,23,24\}\}| = 1,\\
  g_{10,7}(21) &= |\{21,\ldots,27\} \setminus \{21,22,23,24,26,27\}\}| = 1,\\
  g_{10,7}(28) &= |\{28,\ldots,34\} \setminus \{28,\ldots,34\}\}| = 0,\\
  g_{10,7}(35) &= |\{35,\ldots,41\} \setminus \{35,\ldots,41\}\}| = 0.
\end{align*} 
The vector $(g_{10,7}(0),g_{10,7}(3),\ldots,g_{10,7}(35)) = (4,4,3,3,2,2,1,1)$ 
gives the column lengths of the partition $\rho=\gmmap_{7,3}(\pi)=(8,6,4,2)$.
See Figure~\ref{fig:GMmap}.

\begin{figure}[htbp]
  \centering 
      {\scalebox{.4}{\includegraphics{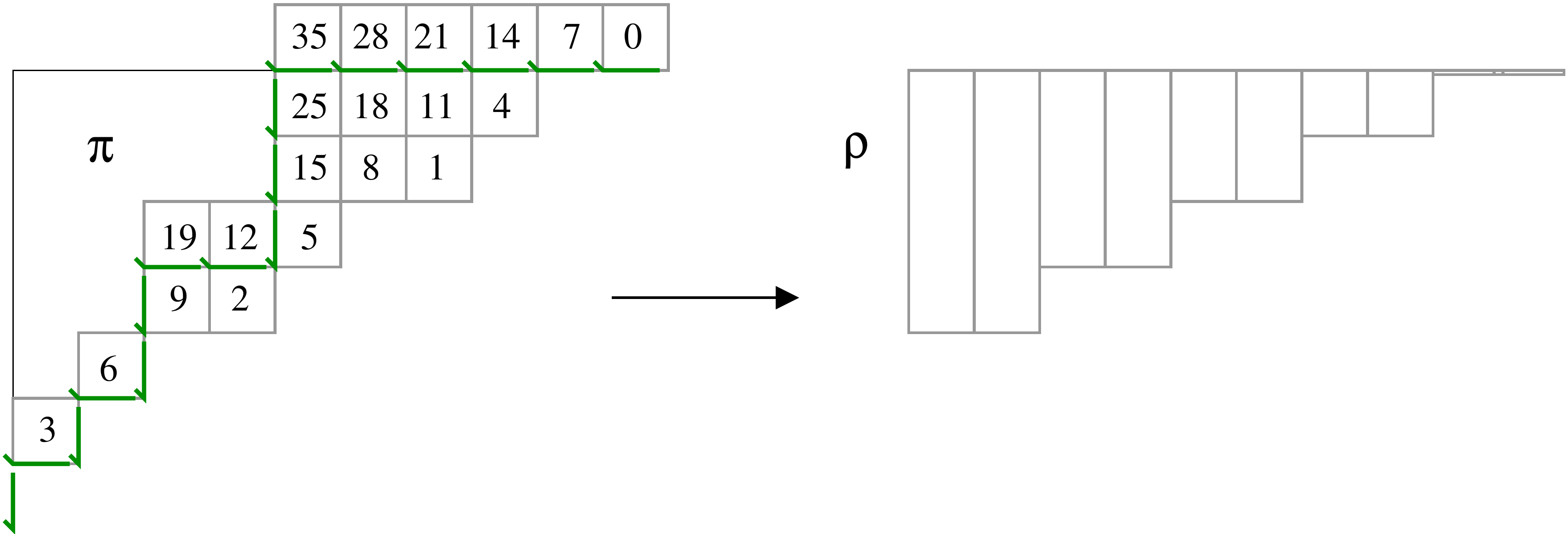}}}
  \caption{Example of the Gorsky-Mazin map $\gmmap_{7,3}$.}
  \label{fig:GMmap}
\end{figure}

The preceding example suggests that the Gorsky-Mazin map coincides
with the map discussed in~\S\ref{subsec:zetamap}. We now prove this fact.

\begin{theorem}\label{thm:gm}
  For $a,b\in\Z_{>0}$ with $\gcd(a,b)=1$ and all $\pi\in \ptnD_{b,-a}(a,b)$,
\[ \mkwd(\gmmap(\pi)) = \sw^{+}_{b,-a}\circ\rev(\mkwd(\pi)).  \]
\end{theorem}

\begin{proof}
  Let $\pi\in\ptnD_{b,-a}(a,b)$ have $b$-generators $\beta_1 < \beta_2
  < \cdots < \beta_b$.  Recall from Table~\ref{tab:sweeps} that
  $\sw^+_{b,-a}\circ\rev$ uses the west-south convention to assign
  levels to steps of a lattice path. It follows that the levels of
  east steps in $\mkpath(\pi)$ are precisely the numbers
  $a+\beta_1<a+\beta_2<\cdots<a+\beta_b$.  The $i$-th east step in the
  output of $\sw^+_{b,-a}\circ\rev$ will be preceded by all the north
  steps whose levels are less than $a+\beta_i$ and followed by all the
  north steps whose levels are greater than $a+\beta_i$. Since the
  output has exactly $a$ north steps total, it will suffice to prove
  (for each fixed $i$) that
\[ \mbox{(the number of north steps of level $<a+\beta_i$)}
  +g_{b,a}(\beta_i)=a. \] 
 For each north step of level $a+k$, the lattice square with level $k$ lies
  below $\mkpath(\pi)$, but all of the squares to the left lie west of
  $\mkpath(\pi)$ and have levels of the form $k+ja$ for some $j\in
  \mathbb{Z}_{\geq 1}$.  On the other hand, for each
  $b$-generator $\beta_i$, $g_{b,a}(\beta_i)$ is the number of levels in the set
  $\{\beta_i,\beta_i+1,\ldots,\beta_i+a-1\}$ that are the levels of
  squares below $\mkpath(\pi)$.  For each north step of level $a+k$ 
  with $k \geq \beta_i$, \[|\{\beta_i,\ldots,\beta_i+a-1\}
  \cap \{k+ja\}_{j > 0}| = 0.\] But for each north step of level $a+k$
  with $k < \beta_i$, then the cardinality will be exactly $1$.  
  Thus the number of levels removed from the $a$-element set
  $\{\beta_i,\ldots,\beta_i+a-1\}$ to obtain $g_{b,a}(\beta_i)$
  is the same as the number of north steps of level $<a+\beta_i$, as needed.  
\end{proof}

\section{Inverting the Sweep Map}
\label{sec:invert-sweep}

\subsection{Introduction.}
\label{subsec:invert-intro}

The main open problem in this paper is to prove that \emph{all sweep maps
are bijections}. Even in the two-letter case, this problem appears to
be very difficult in general. Nevertheless, many special cases of the
sweep map are known to be invertible. After discussing the basic strategy
for inversion (which involves recreating the labels on the output steps
by drawing a suitable ``bounce path''), we describe the inverse sweep
maps that have appeared in the literature in various guises. We omit
detailed proofs that the inverse maps work, since these appear in the
references.  

\subsection{Strategy for Inversion.}
\label{subsec:strategy-invert}

In Figure~\ref{fig:negsweep-ex}, we showed the computation of
$Q=\sw_{3,-2}^-(P)$ where $P,Q$ are paths in $\patR(\N^8\E^{10})$.
The output $Q$ is the path shown on the far right of the figure,
\emph{not including labels}. Suppose we were given $Q$ and needed to
compute $P=(\sw_{3,-2}^-)^{-1}(Q)$. If we could somehow recreate the
labels on the steps of $Q$ (as shown in the figure), then the sweep
map could be easily inverted, as follows. By counting the total number
of north and east steps, we deduce that $P$ must end at level $8\cdot
3 + 10\cdot (-2) = 4$.  We now reconstruct the steps of $P$
in reverse order. The last step of $P$ must be the first step of $Q$ in the
collection of steps labeled 4 (since, when sweeping $P$ to produce
$Q$, level 4 is swept from right to left). We mark that step of $Q$ as
being used. Since it is a north step, the preceding step of $P$ must
end at level 1.  We now take the first unused step of $Q$ labeled 1
(which is a north step), mark it as used, and note that the preceding
step of $P$ must end at level $-2$. We continue similarly, producing
$P$ in reverse until reaching the origin, and marking steps in $Q$ as
they are used.  Because $Q$ is in the image of the sweep map, this
process must succeed (in the sense that all steps of $Q$ are used at
the end, and we never get stuck at some level where all steps in $Q$
with that label have already been used). Evidently, the strategy
outlined here works for any choice of weights, including the general
case of alphabets with more than one letter. Variations of the sweep
map (such as $\sw^+$) can be handled analogously.  The crucial
question is \emph{how to recreate the labels on the steps of $Q$}.

This question has been answered in the literature for Dyck paths,
$m$-Dyck paths, trapezoidal lattice paths, square paths, Schr\"oder
paths, $(n,nm+1)$-Dyck paths, and $(n,nm-1)$-Dyck paths. In every
known case, the key to recreating the labels is to define a
\emph{bounce path} for a lattice path $Q$. The steps of $Q$ associated
with the ``$i$-th bounce'' in the bounce path receive label $i$. Once
labels have been assigned, one can reverse the sweep map as described
in the previous paragraph. We begin by discussing the simplest
instance of the bounce path, which is used to invert the map
$\sw_{1,-1}^-=\phi_{\HL}$ (see~\S\ref{subsubsec:trapz}) acting on Dyck 
paths.  Ironically, several different authors independently introduced
inverse sweep maps even before the map $\phi_{\HL}$ was
proposed in the context of $q,t$-Catalan numbers. We describe 
these inverses in \S\ref{subsec:vaille} and \S\ref{subsec:andrews} below.

\subsection{Inversion of $\phi_{\HL}$ via Haglund's Bounce Path}
\label{subsec:invert-HL}

Figure~\ref{fig:invert-HL} shows the computation of
$Q=\sw_{1,-1}^-(P)=\phi_{\HL}(P)$ for a Dyck path
$P\in\patD(\N^{14}\E^{14})$. To understand how to find
$P=\phi_{\HL}^{-1}(Q)$ given $Q$, it suffices (by the above remarks)
to see how to pass from the unlabeled path $Q$ on the right side of
the figure to the labeled path in the middle of the figure.

\begin{figure}[htbp]
  \centering 
      {\scalebox{.3}{\includegraphics{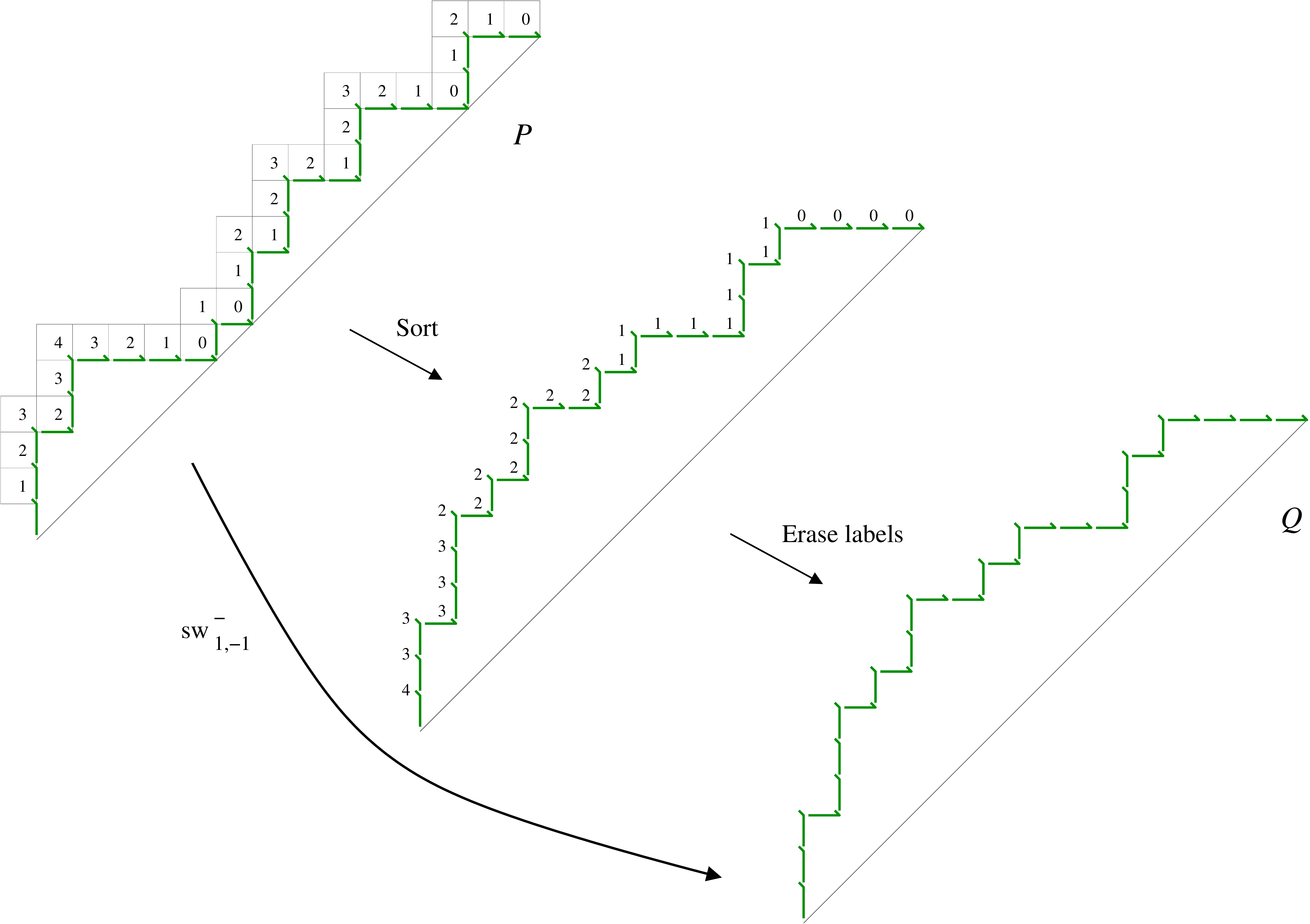}}}
      \caption{Illustration of the action of $\sw^{-}_{1,-1}$ on a
        path $P\in \patD(\N^{14}\E^{14})$.}
  \label{fig:invert-HL}
\end{figure}

To recreate the labels, we draw the \emph{bounce path} for the Dyck path $Q$,
using the following definition due to Haglund~\cite{Hag-bounce}. The bounce
path starts at $(n,n)$ and makes a sequence of horizontal moves $H_i$
and vertical moves $V_i$, for $i=0,1,2,\ldots$, until reaching $(0,0)$.
Each horizontal move is determined by moving west from the current position
as far as possible without going strictly left of the path $Q$.
Then the next vertical move goes south back to the diagonal $y=x$.
Figure~\ref{fig:sweep-bounce} shows the bounce path for our example path $Q$.
In this example, the labels we are trying to recreate are related to
the bounce path as follows: every step of $Q$ located above the bounce
move $H_i$ and to the left of bounce move $V_{i-1}$ has label $i$.
As special cases, the steps of $Q$ above $H_0$ have label zero, and the 
steps of $Q$ to the left of the last bounce move $V_s$ have label $s+1$.
We claim that this relation between the labels and the bounce path holds
in general, for any Dyck path $Q$ of the form $\phi_{\HL}(P)$. This
claim implies that $P$ can be uniquely recovered from $Q$, so that
$\phi_{\HL}$ is injective and hence bijective.

\begin{figure}[htbp]
  \centering 
      {\scalebox{.3}{\includegraphics{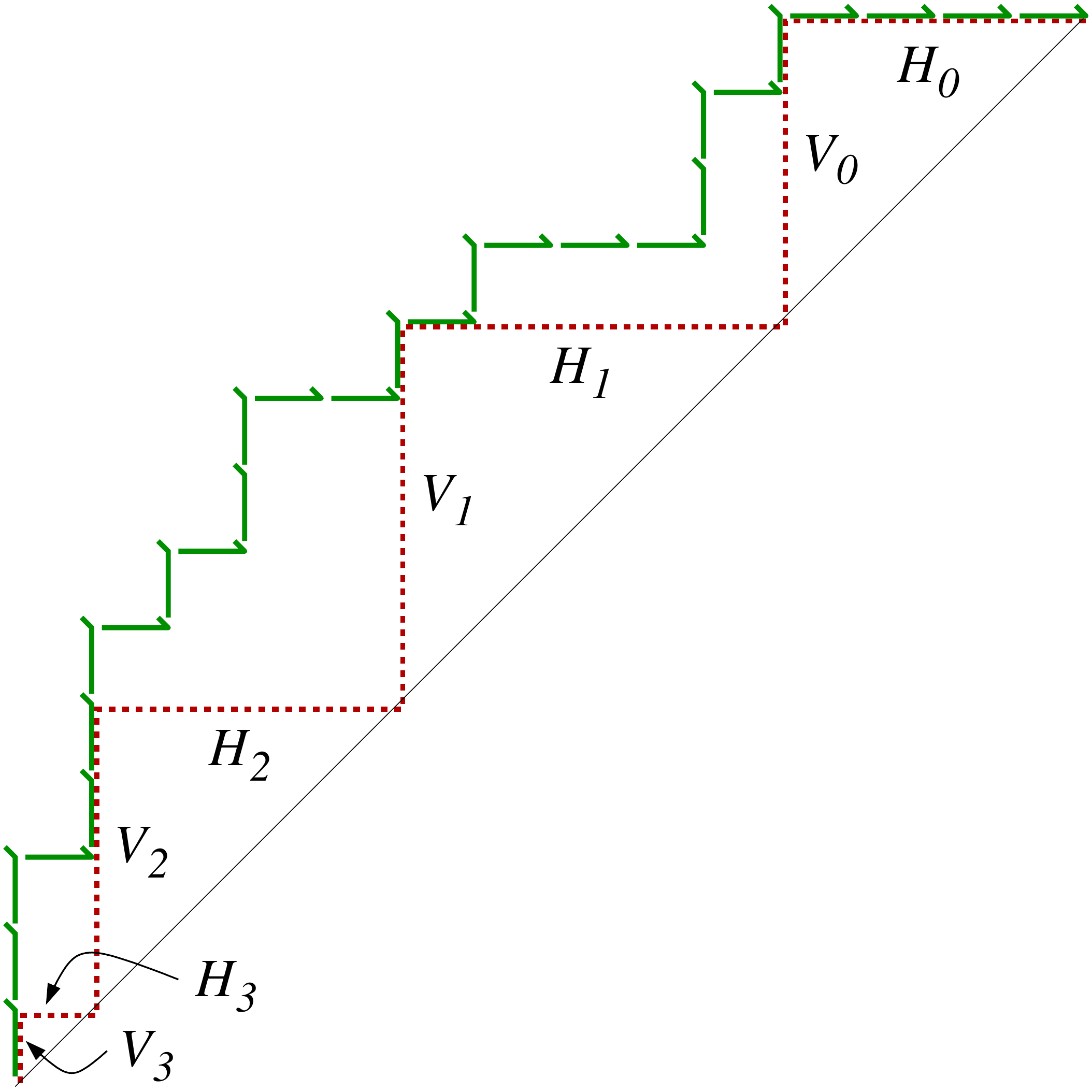}}}
      \caption{Bounce path for path $Q$ of Figure~\ref{fig:invert-HL}.}
  \label{fig:sweep-bounce}
\end{figure}

To prove the claim, let $h_i$ be the length of the horizontal move
$H_i$ of the bounce path for $Q$, and let $v_i=h_i$ be the length of
the vertical move $V_i$ of the bounce path for $Q$. Also define
$v_i=h_i=0$ for $i<0$ and $i>s$, where $s+1$ is the total number of
horizontal bounces.  Finally, let $n_i$ (resp. $e_i$) denote the number of
north (resp. east) steps of $P$ with $(1,-1)$-level equal to $i$.

We first prove this lemma: $n_i = e_{i-1}$ for all $i\in\ZZ$. 
For any level $i$, the number of times the path $P$ arrives at 
this level (via a north step of level $i$ or an east step of level $i$)
equals the number of times $P$ leaves this level
(via a north step of level $i+1$ or an east step of level $i-1$).
This holds even when $i=0$, since $P$ begins and ends at level zero.
It follows that $n_i + e_i = n_{i+1} + e_{i-1}$ for all $i\in\ZZ$.
Also $n_i=0$ for all $i\leq 0$, and $e_i=0$ for all $i<0$.
Thus $n_i=e_{i-1}=0$ for all integers $i\leq 0$. Fix an integer $i\geq 0$,
and assume that $n_i=e_{i-1}$. Then $n_{i+1}=n_i+e_i-e_{i-1}=e_i$,
so the lemma follows by induction.  

Now we show that for all $i\in\Z$, $h_i = e_i$ and $v_{i-1} = n_i$.
Since $P$ is a Dyck path, the assertion holds for all $i<0$.  We prove
the two equalities for $i \geq 0$ by induction on $i$, starting with
the base case of $i=0$.  Recall that the steps of $P$ are swept in
decreasing order by level.  We know that $v_{-1} = n_0 = 0$.  
Since $n_0 = 0$, the path $Q$ ends in $e_0$ east steps.  Hence $h_0
\geq e_0$.  Since $P$ starts at level $0$, the first step in $P$ at
\emph{any} level $i > 0$ must be a north step.  It follows that the
last step in $Q$ labeled with a 1 is a north step and, consequently,
that $h_0 = e_0$.

Assume now that $h_k = e_k$ and $v_{k-1} = n_k$ for some fixed $k\geq
0$.  It follows from the bounce mechanism that $v_k = h_k$.  We
know that $h_k = e_k$ by the induction hypothesis.  Finally, $e_k =
n_{k+1}$ by the above discussion.  Combining these equalities, we find
that $v_k = n_{k+1}$.  We know that $h_{k+1} \geq e_{k+1}$ using
$v_k = n_{k+1}$ and the fact that the east steps of $P$ at level
$k+1$ must be swept after any steps at level $k+2$.  As observed
above, the first step in $P$ at level $k+2$ (if it exists) is a north
step.  Hence $h_{k+1} = e_{k+1}$.

\subsection{Vaill{\'e}'s Bijection}
\label{subsec:vaille}

In 1997, Vaill{\'e}~\cite{vaille} defined a bijection
$\omega$ mapping Dyck paths to Dyck paths, which is the inverse 
of the map $\phi_{n,0,1}$ defined in~\S\ref{subsubsec:trapz}.
(Recall that $\phi_{n,0,1}$ differs from $\phi_{\HL}$ 
by reversing and flipping the output lattice path.) 
Vaill{\'e} gives this example of his bijection 
$\omega$ in~\cite[Fig. 3, p. 121]{vaille}: 
\begin{align*}
         P&=\NE{NNEENNNNNEENNEENEEENNEENNEEENNEE},\\
 \omega(P)&=\NE{NENNENNENNNENNENEEENEEENENNEENEE}.
\end{align*}
The bounce path of $P$ is clearly visible in the left panel of that
figure, although here the bounce path moves from $(0,0)$ north and east 
to $(n,n)$ as a result of the reversal and flipping.

\subsection{The Bijection of Andrews et al.}
\label{subsec:andrews}

Andrews, Krattenthaler, Orsina, and Papi~\cite{AKOP-lie}
described a bijection mapping Dyck partitions to Dyck paths
that is essentially the inverse of $\sw_{1,-1}^{-}$. 
They give an example starting with an input partition
$\pi=(10,10,9,6,5,4,4,3,1,1,1,1,0)$ in~\cite[Fig. 2, p. 3841]{AKOP-lie}.
The word of this partition (after adding one more zero part at the end) is
\[ y=\mkwd(\pi)=\NE{NNENNNNEENENNENENEEENENNEEEE}. \]
This partition maps to the output Dyck path shown 
in~\cite[Fig. 3, p. 3846]{AKOP-lie}, which has word
\[ w=\NE{NENNNEEENNENNEENNNEENEEENNEE}. \]
One may check that $\sw_{1,-1}^-(w)=y$, and similarly for other objects, 
so these authors have inverted the sweep map on Dyck paths. Here too, 
Haglund's bounce path construction (this time proceeding from $(n,n)$ to 
$(0,0)$) is visible in Figure 2 of~\cite{AKOP-lie}.

\subsection{Inverting $\phi_{n,k,m}$ and $\phi'_{n,k,m}$.}
\label{subsec:invert-trapz}

Loehr describes $\phi_{n,0,m}$ and its inverse in~\cite{loehr-mcat}.
The maps $\phi_{n,k,m}$, $\phi'_{n,k,m}$, and their inverses are treated
in~\cite{loehr-trapz}. The key to inversion is defining the bounce
path for a trapezoidal lattice path $Q\in T_{n,k,m}$. This bounce path
starts at $(0,0)$ and moves north and east to $(k+nm,n)$. For $i\geq 0$,
the $i$-th bounce moves north $v_i$ steps from the current location as
far as possible without going strictly north of the path $Q$. The
$i$-th bounce continues by moving east $h_i=v_i+v_{i-1}+\cdots+v_{i-(m-1)}+s$
steps, where $v_j=0$ for $j<0$, $s=1$ for $0\leq i<k$, and $s=0$ for $i\geq k$.
One can show that if $Q$ is produced from $P$ via sweeping
(as described in Step 1 of the proof of Theorem~\ref{thm:trapz-vs-sweep}), 
then the steps of $Q$ located north of the $(i-1)$-th horizontal bounce move 
and west of the $i$-th vertical bounce move receive label $i$.
Thus, we can invert the sweep map in this case.

\subsection{Inverting $\phi_{\LW}$.} 
\label{subsec:invert-LW}

Loehr and Warrington describe the inverse of $\phi_{\LW}$ in~\cite{LW-square}
using the language of area vectors. Their result amounts to inverting the
sweep map $\sw_{1,-1}^-$ on the domain $\patR(\N^n\E^n)$ of lattice paths
in an $n\times n$ square. As usual, it suffices to discuss the construction 
of the ``square bounce path.'' Given $Q=\sw_{1,-1}^-(P)$ with
$P,Q\in\patR(\N^n\E^n)$, first choose the maximum integer $k$ such
that the path $Q$ touches the line $y=x-k$. Call this line the 
\emph{break diagonal}. The \emph{break point} of $Q$ is the lowest
point $(x,y)$ of $Q$ on the line $y=x-k$. The \emph{positive bounce path}
of $Q$ starts at $(n,n)$ and moves to the break point as follows.
First go south $k$ steps from $(n,n)$ to $(n,n-k)$; call this move $V_{-1}$.  
Repeat until reaching the break point, taking $i=0,1,2,\ldots$: In move $H_i$,
go west until blocked by the north end of a north step 
of $Q$; in move $V_i$, go south to the break diagonal. Next, the \emph{negative
bounce path} of $Q$ starts at $(0,0)$ and moves to the break point as follows.
First go east $k$ steps from $(0,0)$ to $(k,0)$; call this move $H_{-1}$.
Repeat until reaching the break point, taking $i=-2,-3,\ldots$:
In move $V_i$, go north to the first lattice point on $Q$; 
in move $H_i$, go east to the break diagonal.
(Since the negative bounce path is blocked by lattice points, not edges, on $Q$, this rule is not the reflection of the rule for the positive bounce path.)
One can check that when sweeping $P$ to produce $Q$,
the steps of $Q$ located north of move $H_i$ and west of move $V_{i-1}$
receive label $i$. It is now straightforward to invert $\phi_{\LW}$.

\subsection{Inverting $\phi_{\EHKK}$.}
\label{subsec:invert-EHKK}

Egge, Haglund, Killpatrick, and Kremer describe the inverse
of their map $\phi_{\EHKK}$ in~\cite[p. 15]{EHKK}. 
Given a Schr\"oder path $Q$, their algorithm 
to compute $P=\phi_{\EHKK}^{-1}(Q)$ 
begins by dividing the steps of $Q$ into regions based on
a version of the bounce path defined for Schr\"oder paths. 
This amounts to reconstructing the labels of the steps of $Q$ 
created when applying the sweep map to $P$.
They then reconstruct the area vector of $P$
(modified to allow diagonal steps) by an insertion process 
that reverses the action of the sweep map. This special case
of sweep inversion is notable because it inverts a sweep map
on a three-letter alphabet (although one of the letters has
weight zero).

\subsection{Inverting $\gmmap_{b,a}$.}
\label{subsec:invert-gorsky-mazin}

In~\cite{GM-jacII}, Gorsky and Mazin describe how to invert
special cases of their map $\gmmap_{b,a}$ using the language of semi-module
generators and bounce paths. Their results amount to inverting
$\sw_{b,-n}$ on the domain $\patD(\N^n\E^b)$ in the case where
$b=nm\pm 1$ for some $m$. The case $b=nm+1$ essentially duplicates the 
$m$-bounce path construction of~\cite{loehr-mcat} (although there is a lot
of new material relating this construction to semi-modules).
The case $b=nm-1$ is a new inversion result obtained via a 
modification of the $m$-bounce paths. Specifically, one constructs
the bounce path's vertical moves $v_i$ (for $i\geq 0$) as
described in~\S\ref{subsec:invert-trapz} above, but now
$h_i=v_i+v_{i-1}+\cdots+v_{i-(m-1)}+t$, where $t=-1$ for $i=m-1$
and $t=0$ for all other $i$.

\section{Area Statistics and Generalized $q,t$-Catalan Numbers}
\label{sec:area-qtcat}

This section applies the sweep map to provide new combinatorial
generalizations of the $q,t$-Catalan numbers~\cite{GH-qtcat} and the
$q,t$-square numbers~\cite{LW-square}.  We make several conjectures
regarding the joint symmetry of these polynomials and their connections
to the nabla operator $\nabla$ on symmetric functions
introduced by A.~Garsia and F.~Bergeron~\cite{nabla1,nabla2,nabla3}.

\subsection{Area Statistics}
\label{subsec:area-stat}

For any word $w\in\{\N,\E\}^*$, let $\area(w)$ be the number of pairs
$i<j$ with $w_i=\E$ and $w_j=\N$. This is the area of the partition
diagram $\mkptn(w)$ consisting of the squares above and to the left of steps 
in $\mkpath(w)$.  
For $r,s\in\Z$, let $w$ have $(r,s)$-levels 
$l_0,l_1,\ldots$ (E-N convention).  Let
$\ml_{r,s}(w)=\min\{l_0,l_1,\ldots\}$, and set
$\area^*_{r,s}(w)=\area(w)+\ml_{r,s}(w)$.  Note that
$\area^*_{r,s}(w)\neq\area^*_{rm,sm}(w)$ in general, so we cannot
necessarily assume that $\gcd(r,s)=1$ when using $\area^*$.

\begin{remark}
  The function $\ml_{b,-a}$ appears in~\cite{ALW-RPF} as $\ml_{b,a}$.
\end{remark}

\begin{remark}
  The correspondence between partitions and paths lying in a fixed triangle
  has led to inconsistent terminology: ``area'' can refer to
  either the number of squares lying in the partition determined by a
  path (as above) or the number of squares between the path and a
  diagonal (such as in~\cite{ALW-RPF}).
\end{remark}

\subsection{Generalized $q,t$-Catalan Polynomials}
\label{subsec:gen-qt-cat}

For $r,s\in\Z$ and $a,b\geq 0$, define the \textbf{$q,t$-Catalan
  numbers for slope $\boldsymbol{(-s/r)}$ ending at
  $\boldsymbol{(b,a)}$} by
\[ C_{r,s,a,b}(q,t)=\sum_{w\in\wdD_{r,s}(\N^a\E^b)}
  q^{\area(w)}t^{\area(\sw_{r,s}^-(w))}. \]

\begin{conj}[Joint Symmetry]~\label{conj:qt-joint}
  For all $r,s\in\Z$ and all $a,b\geq 0$, $C_{r,s,a,b}(q,t)=C_{r,s,a,b}(t,q)$.
\end{conj}

Note that the conjectured bijectivity of $\sw_{r,s}^-$ on the domain
$\wdD_{r,s}(\N^a\E^b)$ would imply the weaker univariate symmetry
property $C_{r,s,a,b}(q,1)=C_{r,s,a,b}(1,q)$.

The rational $q,t$-Catalan polynomials defined in~\cite{ALW-RPF} arise from
a sweep map that, in the case $\gcd(a,b)=1$ considered in that paper,
reduces to $\sw^{+}_{b,-a}\circ\rev$.  As such, the joint symmetry
conjecture~\cite[Conj. 19]{ALW-RPF} is not quite a special case of
Conjecture~\ref{conj:qt-joint}.

Let $w\in\patD_{1,-1}(\N^n\E^n)$ be a ``classical'' Dyck path
with area vector $(g_1,\ldots,g_n)$ (see~\S\ref{subsubsec:trapz}).
Let $\Area(w)=g_1+\cdots+g_n$, which is the number of area 
squares between the path $w$ and the line $y=x$, and let
$\dinv(w)$ be the number of $i<j$ with $g_i-g_j\in\{0,1\}$.
The Garsia-Haiman $q,t$-Catalan numbers~\cite{GH-qtcat} can
be defined by the combinatorial formula
\[ C_n(q,t)=\sum_{w\in\patD_{1,-1}(\N^n\E^n)} q^{\Area(w)}t^{\dinv(w)}. \]
To relate this polynomial to the one defined above, note that
$\area(w)+\Area(w)=n(n-1)/2$. Similarly, it follows from
Theorem~\ref{thm:trapz-vs-sweep} and~\cite[\S2.5]{loehr-mcat} that
$\area(\sw_{1,-1}^-(w))+\dinv(w)=n(n-1)/2$. Therefore, 
\[ C_n(q,t)=(qt)^{n(n-1)/2}C_{1,-1,n,n}(1/q,1/t). \] 
Combining this with a theorem of Garsia and Haglund~\cite{nablaproof}, we get
\[ (qt)^{n(n-1)/2}C_{1,-1,n,n}(1/q,1/t)=\langle\nabla(e_n),s_{(1^n)}\rangle. \]
More generally, for any positive integers $m,n$, the higher-order
$q,t$-Catalan numbers~\cite{loehr-mcat,loehr-thesis} satisfy
\[ C_n^{(m)}(q,t)=(qt)^{mn(n-1)/2}C_{m,-1,n,mn}(1/q,1/t). \] 
The main conjecture for these polynomials can be stated as follows: 
\[ (qt)^{mn(n-1)/2}C_{m,-1,n,mn}(1/q,1/t)
  =\langle\nabla^m(e_n),s_{(1^n)}\rangle. \]
An interesting open problem is to find formulas relating the general 
polynomials $C_{r,s,a,b}(q,t)$ to nabla or related operators.

\subsection{Generalized $q,t$-Square Numbers}
\label{subsec:gen-qt-square}

Next we generalize the $q,t$-square numbers studied in~\cite{LW-square}.
For $a,b\geq 0$, define the \textbf{$\boldsymbol{q,t}$-rectangle numbers
for the $\boldsymbol{a\times b}$ rectangle} by
\[ S_{a,b}(q,t)=\sum_{w\in\wdR(\N^a\E^b)}  
   q^{\area^*_{b,-a}(w)}t^{\area^*_{b,-a}(\sw_{b,-a}^-(w))}. \]

\begin{conj}[Joint Symmetry]
For all $a,b$, $S_{a,b}(q,t)=S_{a,b}(t,q)$.
\end{conj}

The joint symmetry conjecture is known to hold when $a=b$.
This follows from the stronger statement 
\[ (qt)^{n(n-1)/2}S_{n,n}(1/q,1/t)
=2\langle (-1)^{n-1}\nabla(p_n),s_{(1^n)}\rangle, \]
which was conjectured in~\cite{LW-square} and proved
in~\cite{CL-sqthm}.  We conjecture the following more general
relationship between certain $q,t$-rectangle numbers and higher powers
of $\nabla$.

\begin{conj}
For all $m\geq 0$ and $n>0$,
\[ (qt)^{mn(n-1)/2}S_{n,mn}(1/q,1/t)
  =(-1)^{n-1}(m+1)\langle \nabla^m(p_n),s_{(1^n)}\rangle. \]
\end{conj}

\subsection{Specialization at $t=1/q$}
\label{subsec:t=1/q}

Recall the definitions of $q$-integers, $q$-factorials,
and $q$-binomial coefficients: $[n]_q=1+q+q^2+\cdots+q^{n-1}$,
$[n]!_q=[n]_q[n-1]_q\cdots [2]_q[1]_q$, and
${\tqbin{a+b}{a,b}{q}=[a+b]!_q/([a]!_q[b]!_q)}$.

In~\cite[Conj. 21]{ALW-RPF}, the authors make the following conjecture for
coprime $a$ and $b$:
\begin{equation*}
  q^{(a-1)(b-1)/2} \sum_{D\in\wdD(\N^a\E^b)} q^{\area(\sw^{+}_{b,-a}(\rev(D)))-\area(D)}=\frac{1}{[a+b]_q}\dqbin{a+b}{a,b}{q}.  
\end{equation*}
We conjecture here that $\sw^{+}_{b,-a}\circ\rev$ can be replaced by
$\sw^{-}_{b,-a}$:
\begin{conj}
For all coprime $a,b>0$,
\begin{equation*}
  q^{(a-1)(b-1)/2} C_{b,-a,a,b}(q,1/q) = \frac{1}{[a+b]_q}\dqbin{a+b}{a,b}{q}.
\end{equation*}
\end{conj}

We also introduce two conjectures regarding the $t=1/q$ specialization
of paths in a rectangle.  
\begin{conj}
For all $m,n\geq 0$,
\[ q^{m\binom{n}{2}}S_{n,mn}(q,1/q)
  =\frac{(m+1)}{[m+1]_{q^n}}\dqbin{mn+n}{mn,n}{q}. \] 
\end{conj}

This conjecture generalizes to arbitrary rectangles as follows:

\begin{conj}
For all $a,b\geq 0$, write $b=b'k$ and $a=a'k$ for integers $a',b',k\geq 0$
with $\gcd(a',b')=1$. Then
\[ q^{k(a'-1)(b'-1)/2 + a'b'\binom{k}{2}}S_{a,b}(q,1/q)=\frac{(a'+b')}{[a'+b']_{q^k}}
  \dqbin{a+b}{a,b}{q}. \] 
\end{conj}

\begin{remark}
  A Sage worksheet containing code to compute the images of paths
  under various versions of the sweep map as well as to check the
  conjectures of Section~\ref{sec:area-qtcat} can be found at the
  third author's web page~\cite{sage-worksheet}.
\end{remark}

\section*{Acknowledgments}

The authors gratefully acknowledge discussions with Jim Haglund, Mark
Haiman and Michelle Wachs.


\end{document}